\documentclass[11pt,a4paper]{article}

\usepackage[T1]{fontenc}

\usepackage[latin9]{inputenc}
\usepackage{amscd}
\usepackage{mathtools}
\usepackage{lmodern}
\usepackage{amsthm,amsmath,amsfonts,amssymb,calrsfs,thmtools}
\usepackage{graphicx}
\usepackage{enumerate}
\usepackage[shortlabels]{enumitem}
\usepackage[colorlinks=true,hyperindex=true]{hyperref}

\let\OLDthebibliography\thebibliography
\renewcommand\thebibliography[1]{
  \OLDthebibliography{#1}
  \setlength{\parskip}{0pt}
  \setlength{\itemsep}{0pt plus 0.3ex}
}

\usepackage[activate={true,nocompatibility},final,tracking=true,kerning=true,spacing=true,factor=1100,stretch=15,shrink=15]{microtype}
\microtypecontext{spacing=nonfrench}

\setlist[enumerate]{topsep = 0.5ex, leftmargin=1.5cm, itemsep = -3pt}
\setlist[itemize]{topsep = 0.5ex, leftmargin=1cm, itemsep = -3pt}

\newtheorem{thm}{Theorem}[section]{\bf}{\it}
\newtheorem{cor}[thm]{Corollary}{\bf}{\it}
\newtheorem{lem}[thm]{Lemma}{\bf}{\it}
\newtheorem{prop}[thm]{Proposition}{\bf}{\it}
{\it}{\rm}

\newtheorem{theorem}{Theorem}

\newtheorem{lemA}[theorem]{Lemma}

\newcommand{\abs}[1]{\left\lvert #1 \right \rvert}
\newcommand{\intpart}[1]{\lfloor #1 \rfloor}

\newcommand{\norm}[1]{\lVert #1 \rVert}

\newcommand{\mc}[1]{\mathcal{#1}}
\newcommand{\m}[1]{\mathbb{#1}}

\def\ie{i.e. }

\renewcommand\Re{\operatorname{Re}}
\renewcommand\Im{\operatorname{Im}}

\def\QS{\operatorname{QS}}

\def\a{\alpha}
\def\b{\beta}
\def\g{\gamma}
\def\G{\Gamma}
\def\d{\delta}
\def\D{\Delta}

\def\t{\theta}

\def\l{\lambda}
\def\k{\kappa}

\def\s{\sigma}
\def\S{\Sigma}

\def\vare{\varepsilon}
\def\HH{{\mathbb H}}
\def\Chat{\hat{\m{C}}}
\def\eps{\varepsilon}

\def\dd{\,d}
\def\vol{\mathrm{vol}}

\newcommand{\ad}[1]{\overline{#1}}
\newcommand{\inprod}[2]{\langle #1,#2 \rangle}

\def\detz{\mathrm{det}_{\zeta}}
\def\P{\mathrm P}

\begin{document}

\title{Equivalent Descriptions of the Loewner Energy}

\author{Yilin Wang \thanks{Department of Mathematics, 
 ETH Zurich, Switzerland.  Email: yilin.wang@math.ethz.ch}}

\date{April 18, 2019}


\maketitle

\begin{abstract}
Loewner's equation provides a way to encode a simply connected domain or equivalently its uniformizing conformal map via a real-valued driving function of its boundary. 
The first main result of the present paper is that the Dirichlet energy of this driving function (also known as the Loewner energy) is equal to the Dirichlet energy of the log-derivative of the (appropriately defined) uniformizing conformal map.

This description of the Loewner energy then enables to tie direct links with regularized determinants and Teichm\"uller theory: 
We show that for smooth simple loops,  the Loewner energy can be expressed in terms of the zeta-regularized determinants of a certain Neumann jump operator. 
We also show that the family of finite Loewner energy loops coincides with the Weil-Petersson class of quasicircles,
and that the Loewner energy equals to a multiple of the universal Liouville action introduced by Takhtajan and Teo, which is a K\"ahler potential for the Weil-Petersson metric on the Weil-Petersson Teichm\"uller space. 
\end{abstract}
\section{Introduction}
   
   \subsection*{\bf Background on Loewner energy} 
   Loewner introduced in 1923 \cite{Loewner1923} a way to encode/construct uniformizing conformal maps, via continuous iterations (now known as Loewner chains) of simple infinitesimal conformal distortions. It allows to  describe the uniformizing maps via a real-valued function that is usually referred to as the driving function of the Loewner chain.  
   Loewner's motivation came from the Bieberbach conjecture and Loewner chains have in fact been an important tool in the proof 
   of this conjecture by De Branges \cite{DeBranges1985} in 1985. 
   They are also a fundamental building block in the definition of Schramm-Loewner Evolutions by Schramm \cite{schramm2000scaling}. 
   
   Let us very briefly recall aspects of the Loewner chain formalism in the chordal setting, which is the first one that we will focus on here:  
   When $\g$ is a simple curve from $0$ to $\infty$ in the upper half-plane $\m H$, one can choose to parametrize $\g$ in a way such that the half-plane capacity of $\g [0,t]$ seen from infinity grows linearly. More precisely, this means that the \emph{mapping-out function} $g_t$ from $\m H \setminus \g[0,t]$ to $\m H$, that is normalized near infinity by $g_t (z) = z + o(1)$ as $z\to \infty$ does in fact satisfy $g_t (z) = z + 2t/z + o(1/z)$.
   By Carath\'eodory's theorem, the function $g_t$ can be extended continuously to the tip $\g_t$ of the slit $\g[0,t]$ which enables to define 
   $W(t) := g_t (\g_t)$. 
   The real-valued function $W$ turns out to be continuous, and it is called the \emph{driving function} of the chord $\g$ (or of the Loewner flow $(g_t)_{t \ge 0}$) in $(\m H, 0, \infty)$. 
   Loewner  \cite {Loewner1923} showed (in the slightly different radial setting, but the story is essentially the same as in this chordal setting, see 
   \cite{pommerenke_Loewner}) that the functions $t \mapsto g_t (z)$ do satisfy a very simple differential equation, which in turn 
   shows that the driving function uniquely determines the curve $\g$. Note that when $W$ is only continuous, it may not arise from a curve $\g$, however, the Loewner flow $g_t$ is always well defined on a subset of $\m H$.
   
The random curves driven by $W(t) := \sqrt \k B_t$, where $\k >0$ and $B$ is one-dimensional Brownian motion are Schramm's chordal SLE$_\kappa$ \cite{schramm2000scaling}
(it can be shown that these random curves are almost surely simple curves when $\kappa \le 4$ \cite{BasicSLE} and in that case, we are
in the framework described above),
which is conjectured (and for some special values of $\k$, this is proven -- see \cite {LSW,harmonicexplorer,smirnov}) to be the scaling limit of interfaces in some statistical physics models.
Given that the action functional that is naturally related to Brownian motion is the Dirichlet energy $\int_0^{\infty} W'(t)^2 /2  \dd t$,
this energy looks like a natural quantity to investigate in the Loewner/SLE context. 
It has in fact been introduced and studied recently by Friz and Shekhar \cite{FrizShekhar2017} and the author \cite{wang_loewner_energy} independently.
It should be emphasized that this Loewner energy is finite only when the simple chord $\g$ is quite regular, and that we will therefore be dealing only with fairly regular chords as opposed to SLEs in the present paper. 
Since a simple chord determines its driving function, one can view this Loewner energy as a function of the chord and 
denote it by $I_{\m H, 0,\infty}(\g)$. 

Elementary scaling considerations show that for any given positive $\lambda$, 
$$I_{\m H, 0,\infty}(\lambda \g)=I_{\m H, 0,\infty}(\g).$$
 This enables to 
 define the Loewner energy $I_{D,a,b}$ of a simple chord in a simply connected domain $D$ from $a$ to $b$ (where $a$ and $b$ are distinct prime ends of $D$), to be the Loewner energy of the image of this  chord in $\m H$ from $0$ to $\infty$ under any  uniformizing map from $D$ to $\m H$ which maps $a$ and $b$ to $0$ and $\infty$.

In our paper \cite {wang_loewner_energy}, we have shown that this Loewner energy was reversible, namely that 
$I_{D, a,b}(\g) = I_{D, b,a}(\g)$. Even though this is a result about deterministic Loewner chains, our proof was based on the reversibility of SLE$_\k$ and on an interpretation of the Loewner energy as a large deviation functional for SLE$_\k$ as $\k \to 0+$. This result
raised the question whether there are direct descriptions of the Loewner energy that do not involve the underlying Loewner chains. 
The goal of the present paper is to provide such descriptions. In fact, we will provide three such expressions of the Loewner energy, which we now briefly describe in the next three paragraphs. 

\subsection*{\bf Relation to the Dirichlet energy of the log-derivative of a uniformizing map} 
Let us introduce some notation: It turns out to be more convenient to work in a slit plane rather than in the 
upper half-plane (this just corresponds to conjugation of $g_t$ by the square map) when studying the Loewner energy of chords. 
In other words, one looks at a chord from $0$ to $\infty$ in the slit plane $\S:= \m C \setminus \m R_+$. 
Such a chord divides the slit plane into two connected components $H_1$ and $H_2$, and one can then define $h_i$ to be a conformal map from $H_i$ onto a half-plane fixing $\infty$. 
  See Figure \ref{fig_thm_chord} for a picturesque description of these two maps. 
Let $h$ be the map defined on $\S \setminus \g$, which coincides with $h_i$ on $H_i$.
Here and in the sequel $\dd z^2$ denotes the Euclidean (area) measure on $\m C$. 

\begin {thm}
\label {thm_chord_int_as_thm}
When $\g$ is a chord from $0$ to infinity in $\S$ with finite Loewner energy, then
\[I_{\S, 0, \infty} (\g)= \frac{1}{\pi} \int_{\S \setminus \gamma} \abs{\nabla \log \abs{h'(z)}}^2 \dd z^2 = \frac{1}{\pi} \int_{\S \setminus \gamma} \abs{\frac{h''(z)}{h'(z)}}^2 \dd z^2.\]
\end {thm}

 \begin{figure}[ht]
 \centering
 \includegraphics[width=0.7\textwidth]{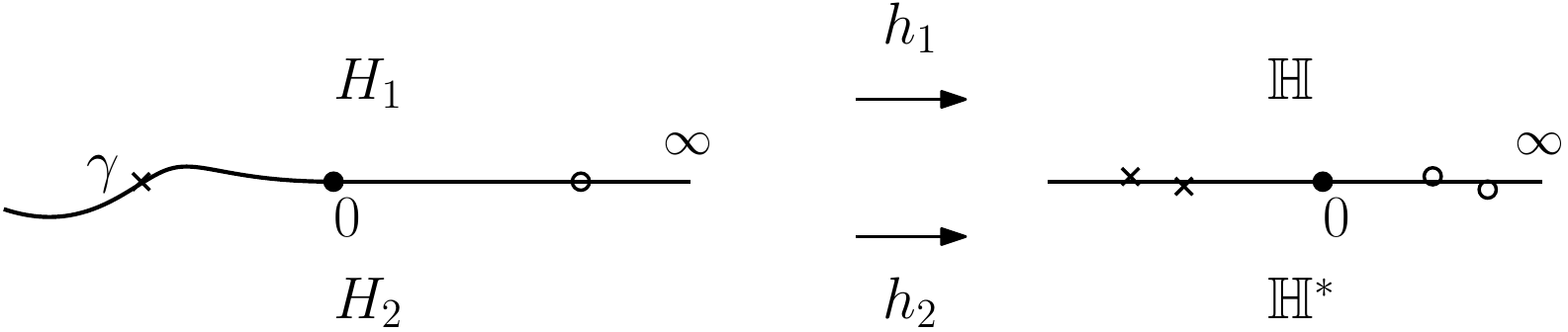}
 \caption{\label{fig_thm_chord} We often choose the half-planes to be $\m H$ and the lower half-plane $\m H^*$ as the image of $h_1$ and $h_2$ to fit into the Loewner chain setting. However, it is clear that the last two expressions of the equality in Theorem~\ref{thm_chord_int_as_thm} is invariant under transformations $z \mapsto az +b$, for $a \in \m C^*$ and $b\in \m C$.  
} 
 \end{figure}

The Loewner energy also has a natural generalization to oriented simple loops (Jordan curves) with a marked point (root) embedded in the Riemann sphere, such that if we identify the simple chord $\g$ in $\S$ connecting $0$ to $\infty$ with the loop $\g \cup \m R_+$, then the loop energy of $\g \cup \m R_+$ rooted at $\infty$ and oriented as $\g$ is equal to the chordal Loewner energy of $\g$ in $(\S, 0, \infty)$. 
In a joint work with Steffen Rohde \cite{LoopEnergy}, we have shown that this Loewner loop energy, denoted by $I^L$, depends only on the image (i.e. of the trace) of the loop. 
In particular, it does not depend on the root of the loop. 
The Loewner loop energy is therefore a non-negative and M\"obius invariant quantity on the set of free loops, which vanishes only on circles.
The proof in \cite{LoopEnergy} relies on the reversibility of chordal Loewner energy and a certain type of  surgeries on the loop to displace the root.  
This root-invariance suggests that the framework of loops provides even more symmetries and invariance properties than the chordal case when one studies  Loewner energy. 

In the present paper, we will derive the counterpart of Theorem~\ref{thm_chord_int_as_thm} for loops:
\begin{thm}[see Theorem~\ref{thm_loop_identity}]
\label {thm_mr2}
 If $\g$ is a loop passing through $\infty$ with finite Loewner energy, then 
\[I^L(\g) = \frac{1}{\pi}  \int_{\m C \setminus \g} \abs {\nabla \log \abs{h' (z)}}^2 \dd z^2 ,\]
  where $h$ maps $\m C \setminus \g$ conformally onto two half-planes and fixes $\infty$.
\end{thm}
Actually, one can view Theorem~\ref {thm_chord_int_as_thm} as a consequence of Theorem~\ref{thm_loop_identity} (and this is the order in which we will derive things).
Note that the identity also holds when $I^L(\g) = \infty$, which follows in fact from the characterization of Weil-Petersson quasicircles (see below) by its welding homeomorphism \cite{Shen2017} that we will not enter into detail here.

Let us say a few words about the strategy of our proof of these two results, which will be the main purpose of the first part of the present paper (Sections \ref{sec_weak_J} to \ref{sec_loop}).
We will first derive the additivity (called $J$\emph{-additivity}) of the integral on log-derivative of the uniformizing map when the curve is $C^{1,\a}$-regular (Section~\ref{sec_weak_J}).
Curves with piecewise linear driving function fall into this class of curves. 
Weak $J$-additivity suffices to obtain a version of Theorem~\ref{thm_chord_int_as_thm} for finite capacity chords driven by a piecewise linear function using explicit infinitesimal computation in Section~\ref{sec_linear}. 
It then provides the bound to deduce the general $J$-additivity for all finite energy curves (Corollary~\ref{cor_strong_J_add}) and the proof of the identity (Proposition~\ref{prop:eq_chord_1}) for finite capacity and finite energy chords is completed in Section~\ref{sec_approx_finite_chord}. 
We prove Theorem~\ref{thm_loop_identity} in Section~\ref{sec_loop} by passing the capacity to $\infty$ and generalize the identity to loops.
It is worth emphasizing that already in the case of a linear driving function where the map $h$ is almost explicit, the proof is not immediate.

As briefly argued in the concluding section (Section \ref {informal}) of the present paper,  it is possible to heuristically interpret  Theorem~\ref{thm_loop_identity}
as a $\kappa \to 0+$ limit of some relations between SLE$_\kappa$
curves and Liouville Quantum gravity, as pioneered by Sheffield in \cite {scott_zipper}. This is actually the line of thought that led us to guessing the Theorem~\ref{thm_loop_identity}.

Theorem~\ref{thm_loop_identity} then opens the door to a number of connections with other ideas, which we then investigate in Section~\ref{sec_detz} and Section~\ref{sec_WP} and that we now describe.

\subsection*{\bf Relation to zeta-regularized determinants}
The first approach involves zeta-regularized determinants of Laplacians for smooth loops. Our main result in this direction is Theorem~\ref{thm_energy_determinant}, which can be summarized by: 
\begin {thm}
\label {thm_mr3}
For $C^\infty$ loops, one has the 
identity 
$$I^L(\g) = 12 \log \detz' N(\g, g) - 12 \log l_g(\g) - (12 \log \detz' N (S^1, g) - 12 \log l_g(S^1) ),$$
where $g$ is any metric on the Riemann sphere conformally equivalent to the spherical metric, $l_g(\g)$ the arclength of 
$\g$, and $\detz' N(\g, g)$ the zeta-regularized determinant of the Neumann jump operator across $\g$.
\end {thm}

Let us already note that the root-invariance (and also the reversibility) 
of the loop energy for smooth loops follows directly from this result, because there is no more parametrization involved in the right-hand side. 
This identity is also reminiscent of the partition function formulation of the SLE/Gaussian free field coupling by Dub\'edat \cite{Dubedat_GFF}. 

The zeta-regularization of operators was introduced by Ray and Singer \cite{RaySinger1971} and are then used by physicists (e.g. Hawking \cite{Hawking1977}) to make sense of quadratic path integrals. 
The determinants of Laplacians on Riemann surfaces also play a crucial role in Polyakov's quantum theory of bosonic strings \cite{polyakov1981B}. 
Polyakov and Alvarez studied the variation of the functional integral under conformal changes of metric, for surfaces with or without boundary \cite{Alvarez1983}, resp. \cite{polyakov1981B} which is known as the Polyakov-Alvarez conformal anomaly formula (Theorem~\ref{thm_polyakov}). 
Osgood, Phillips and Sarnak \cite{OPS} showed that such variation is realized by the zeta-regularized determinants of Laplacians. 
The Polyakov-Alvarez conformal anomaly formula is the main tool in our proof of Theorem~\ref{thm_energy_determinant}.
Notice that the regularization is well defined when the boundary of the bounded domain is regular enough (e.g. $C^2$), and that the variation formula was also derived under boundary regularity conditions.
This is why to stay on the safe side in the present paper, we restrict ourselves to $C^{\infty}$ loops whenever we consider these regularized determinants to conform with the setup of \cite{BFK1992MV,BFKM1994logdet,OPS}.

The zeta-regularized determinant of the Neumann jump operator $N(\g, g)$ that is referred to in Theorem \ref {thm_mr3} is 
closely related to such determinants of Laplacians via a Mayer-Vietoris type surgery formula \cite{BFK1992MV} that we will recall in Section \ref {sec_detz}.

\subsection*{\bf Relation to the Weil-Petersson Teichm\"uller space} 
It was shown in \cite{LoopEnergy} that finite energy loops are quasicircles. 
Since the Loewner energy is M\"obius invariant, it is very natural to consider them as points in the universal Teichm\"uller space $T(1)$ which can be modeled by the homogeneous space 
$\text{M\"ob} (S^1) \backslash \QS(S^1)$ 
that is the group QS$(S^1)$ of quasisymmetric homeomorphisms of the unit circle $S^1$ modulo M\"obius transformations of $S^1$, via the welding function of the quasicircle (for basic material on quasiconformal maps and Teichm\"uller spaces, readers may consult e.g. \cite{lehto2012univalent,lehto1973quasiconformal,Nag1988complex}). 
On the other hand, it is easy to see that quasicircles do not always have finite Loewner energy (for instance, quasicircles with corners have infinite energy). 
This raises the natural question to identify the subspace of finite energy loops in the Teichm\"uller space. The answer to this question is the main purpose of Section~\ref{sec_WP}.
 
Recall that the equivalent classes  $\text{M\"ob}(S^1) \backslash \text{Diff}(S^1)$ of smooth diffeomorphisms of the circle is naturally embedded into $T(1)$ since they are clearly quasisymmetric. 
It carries a remarkable complex structure, and there is a unique (up to constant factor) homogeneous K\"ahler metric on it which has also been studied intensively by both physicists and mathematicians, see e.g. Bowick, Rajeev, Witten \cite{BowickRajeev1987holomorphic,BowickRajeev1987string,Witten} as it plays an important role in the string theory. 
Nag and Verjovsky \cite{Nag1988complex} showed that this metric coincides with the Weil-Petersson metric on $T(1)$ and Cui \cite{Cui2000} showed that the completion $T_0(1)$ (called the Weil-Petersson Teichm\"uller space) of $\text{M\"ob} (S^1)\backslash \text{Diff}(S^1)$ under the Weil-Petersson metric is the class of quasisymmetric functions whose quasiconformal extension has $L^2$-integrable complex dilation with respect to the hyperbolic metric. 

The memoir by Takhtajan and Teo \cite{TT2006WP} studies systematically the Weil-Petersson Teichm\"uller space. They proved that $T_0(1)$ is the connected component of the identity in $T(1)$ viewed as a complex Hilbert manifold (this is actually where the notation of $T_0(1)$ comes from) and established many other equivalent characterizations of the Weil-Petersson Teichm\"uller space. 
They also introduced a quantity which is very relevant for the present paper: the universal Liouville action $\bf S_1$ (we will recall its definition in \eqref{eq_def_S1}) and showed that it is a K\"ahler 
potential for the Weil-Petersson metric on $T_0(1)$. Later, Shen, et al. \cite{Shen2013,Shen2017,Shen2018} did characterize $T_0(1)$ directly in terms of the welding homeomorphisms.  

The main result of Section~\ref{sec_WP} of the present paper is Theorem~\ref{thm_energy_liouville} that loosely speaking says  that:
\begin {thm}
\label {thm_mr4}
A Jordan curve $\g$ has finite Loewner energy if and only if $[\g] \in T_0(1)$ and 
$$I^L(\g) = {\bf S_1} ([\g]) / \pi,$$
where we identify $\g$ with its welding function which lies in $\QS(S^1)$. 
\end {thm}
This provides therefore another characterization of $T_0(1)$ and a new viewpoint on its K\"ahler potential (or alternatively a way to look at the Loewner energy). 

Again the root-invariance (and also the reversibility) of the loop energy can be viewed as a corollary of this result, because there is no more parametrization involved in the definition of ${\bf S_1} ( [\g] )$. Note that we require no regularity assumption on $\g$ in the above identity.

\medbreak
The paper is structured as follows: Section~\ref{sec_weak_J} to Section~\ref{sec_loop} are devoted to the proof of Theorem~\ref{thm_loop_identity} as we described above, from which we derive in Section~\ref{sec_detz} the identity with determinants (Theorem~\ref{thm_energy_determinant}) for smooth loops. In Section~\ref{sec_WP}, by choosing a particular metric in the identity of Theorem~\ref{thm_energy_determinant}, we deduce Theorem~\ref{thm_energy_liouville} which relates the Loewner energy to the the Weil-Petersson Teichm\"uller space, via approximation of finite energy curves by smooth curves. In Section~\ref{informal} we gather informal discussions on how we are led to Theorem~\ref{thm_loop_identity}.
\medbreak

{\bf Acknowledgements}  
   I would like to thank Wendelin~Werner for numerous inspiring discussions as well as his help during the preparation of the manuscript. I also thank Steffen Rohde, Yuliang Shen, Lee-Peng Teo, Thomas Kappeler, Alexis Michelat and Tristan Rivi\`ere for helpful discussions, and the referee for many constructive comments. This work is supported by the Swiss National Science Foundation grant \# 175505.

   \section{Preliminaries and notation} \label{sec_prelim}
   
   Through out the paper, a \emph{domain} means a simply connected open subset of $\m C $ whose boundary can be parametrized by a non self-intersecting continuous curve (not necessarily injective). We orient and parametrize this boundary so that it winds anti-clockwise around the domain.
When the boundary is a Jordan curve then we say that the domain is a \emph{Jordan domain}.

We first recall that a real-valued function $f$ defined on the compact interval $ [a,b]$ is absolutely continuous (AC) if 
there exists a Lebesgue integrable function $g$ on $[a,b]$, such that
        \[f(x) = f(a) + \int_a^x g(t) \dd t, \quad \text{for } x \in [a,b].\]
It is elementary to check that this is equivalent to any of the following two conditions (see \cite{AthereyaLahiri}~Sec.~4.4):
  \begin{enumerate}[(\text{AC}1),leftmargin=1cm]
     \item \label{AC_interval} For every $\vare>0$, there is $\d >0$ such that whenever a finite sequence of pairwise disjoint sub-intervals $(x_{k},y_{k})$ of $[a,b]$ and
    $\sum_{k}(y_{k}-x_{k})<\d$,
then
\[ \sum_{k}|f(y_{k})-f(x_{k})|<\vare.\]
     \item \label{AC_derivative} $f$ has derivative almost everywhere, the derivative is Lebesgue integrable, and 
     \[f(x) = f(a) + \int_a^x f'(t) \dd t, \quad \text{for } x \in [a,b].\]
  \end{enumerate}
 A function $f$ defined on a non-compact interval is said to be AC if $f$ is AC on all the compact sub-intervals.

 We now generalize the definition of the Loewner energy of a chord $\g$ in $(D,a,b)$ that we gave in the introduction to the case of chords that start at $a$ 
 but do not make it all the way to $b$. The steps of the definition are exactly the same: 
 \begin {itemize} 
  \item First, consider the case of the upper half-plane, and consider a finite simple chord $\g := \g [0,T]$, parametrized by its 
  half-plane capacity. We then let $W$ be the driving function of the chord, and we 
  set $$
I_{\HH,0,\infty} (\g  [0,T] ) :=
     \frac{1}{2} \int_0^T W'(t)^2 \dd t \quad  \text{ when $W$ is absolutely continuous}
$$ 
and $
I_{\HH,0,\infty} (\g  [0,T] ) = \infty$ if $W$ is not AC.
Sometimes with a slight abuse of notation, we denote also the above quantity by $I(W)$.
\item We note that this definition of the energy of the chord $\gamma [0,T]$ is invariant under scaling,  
 so that for any conformal map $\phi$ from $\HH$ onto some simply connected domain $D$, 
 we can define $$I_{D, a, b} ( \phi \circ \g [0,T] ) := I_{\HH, 0, \infty} (  \g [0,T] ),$$
 where $a = \phi (0)$ and $b= \phi (\infty)$. 
 \end {itemize}

Let us list some other properties of finite Loewner energy curves: If $\g$ has finite energy in $(D,a,b)$ and is parametrized by capacity 
(the parametrization pulled back by the uniformizing map $\phi$), then
\begin{itemize}
  \item $I_{D,a,b} (\g) = 0$ if and only if $\g$ is contained in the conformal geodesic in $D$ from $a$ to $b$, \ie $ \g = \phi^{-1} (i[0,s])$ for some $s \in [0,\infty]$.
  \item $\g$ is a rectifiable simple curve, see \cite[Thm.~2.iv]{FrizShekhar2017}. 
  \item $\g$ is a quasiconformal curve, that is the image of the conformal geodesic in $D$ between $a$ and $b$ under a quasiconformal map from $D$ onto itself fixing $a$ and $b$. 
  In particular,  $b$ is the only boundary point hit by $\g_t$ and it happens only when $t = T =  \infty$, see \cite[Prop.~2.1]{wang_loewner_energy}. 
  \item $I$-\emph{Additivity}: Since  $\forall t \leq T$,
  $$\frac{1}{2} \int_0^T W'(r)^2 \dd r  =  \frac{1}{2} \int_0^t W'(r)^2 \dd r  + \frac{1}{2} \int_t^T W'(r)^2 \dd r, $$
  it follows from the definition of the driving function and the Loewner energy that
  \begin {equation}
   \label {IAdd}
I_{D,a,b} (\g [0,T]) = I_{D,a,b} (\g [0,t]) + I_{D\setminus \g[0,t], \g_t, b} (\g[t,T]).
\end {equation}
In particular, if $W$ is constant on $[t,T]$, $\g[t,T]$ is contained in the conformal geodesic of $D \setminus \g[0,t]$ from $\g_t$ to $b$.
  \item $\g$ has no corners, see \cite[Sect.~2.1]{LoopEnergy}.
  \item $\g$ need not to be $C^1$, see the example of slow spirals in \cite[Sect.~4.2]{LoopEnergy}.
\end{itemize}

From now on in this section, 
we restrict ourselves in the domain $(D,a,b) = (\S, 0, \infty)$ where $\S = \m C \setminus \m R_+$. We will abbreviate $I_{(\S, 0, \infty)}$ as $I$. We choose $\sqrt {\cdot}$, the square root map taking
values in the upper half-plane $\m H$, to be the uniformizing conformal map of $(\S, 0, \infty)$, so that the capacity of a bounded hull in $\S$, as well as the driving function of Loewner chains in $(\S, 0, \infty)$ are well-defined (and not up to scaling any more).

The following result is the counterpart of Theorem~\ref {thm_chord_int_as_thm} for chords that do not make it all the way to infinity (\ie $T <\infty$):  

\begin{prop}\label{prop:eq_chord_1}

Let $\g$ be a finite energy simple curve in $(\S, 0, \infty)$,
    \begin{equation}
   I(\g [0,T] ) = \frac{1}{\pi} \int_{\S \setminus \g [0,T]}\abs{\frac{h_T''(z)}{h_T'(z)}}^2 \dd z^2,
    \end{equation}
    where $h_T : \S \setminus  \g \to \S $ is the conformal mapping-out function of $\g [0,T]$, such that $h_T(\g_T) = 0$ and $h_T (z) = z + O(1)$ as $z \to \infty$. 
    \end{prop}

Note that Proposition~\ref{prop:eq_chord_1} is weaker than Theorem~\ref{thm_chord_int_as_thm}. Indeed, if we consider $W$ as in 
Proposition~\ref{prop:eq_chord_1} and then defines $\tilde W$ on all of $[0, \infty)$ by $\tilde W(t)  := W ( \min (t, T))$, then $\tilde W$ does generate the chord 
$\ad \g$ from $0$ to infinity in $\S$ that coincides with $\g$ up to time $T$ and then continues along the conformal geodesic from $\g_T$ to infinity in $\S \setminus 
\g [0,T]$ (see Figure~\ref{fig_extension}). 

 \begin{figure}[ht]
 \centering
 \includegraphics[width=0.7\textwidth]{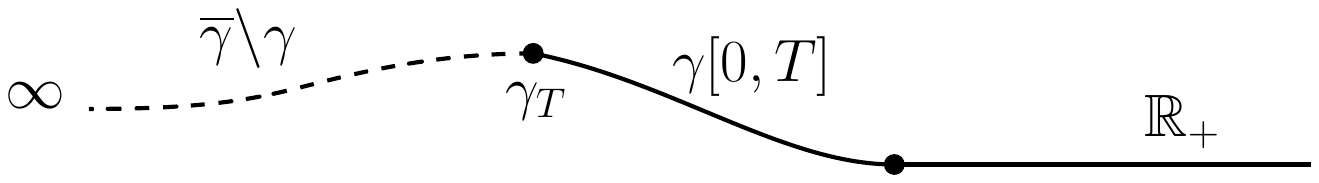}
 \caption{\label{fig_extension} The infinite capacity curve $\ad \g$ is the completion of $\g$ by adding the conformal geodesic $\ad \g \setminus \g = h_T^{-1} (\m R_-)$ connecting $\g_T$ to $\infty$ in $\S \setminus \g[0,T]$.
} 
 \end{figure}

It is easy to see that the restriction of $h_T$ to $\S \setminus \ad \g$ is an admissible choice for the conformal map in Theorem~\ref {thm_chord_int_as_thm}, which maps $\S \setminus \ad \g$ to two half-planes, so that Proposition \ref{prop:eq_chord_1} is 
a rephrasing of Theorem~\ref {thm_chord_int_as_thm} for $\ad \gamma$. 
However, we will explain how it is in fact possible to deduce Theorem~\ref{thm_chord_int_as_thm}  from Proposition~\ref{prop:eq_chord_1} by letting $T \to \infty$ in Section~\ref{sec_loop} while proving the more general Theorem~\ref{thm_loop_identity} for simple loops. 
We will therefore 
aim at establishing Proposition~\ref{prop:eq_chord_1} which is completed in Section~\ref{sec_approx_finite_chord}. 

\medbreak 
In the sequel we will denote the right-hand side of Proposition~\ref{prop:eq_chord_1} by $J(h_T)$.
Note that
$$J (h_T)= \frac{1}{\pi} \int_{\S\setminus \g} \abs{\frac{h_T''(z)}{h_T'(z)}}^2 \dd z^2=  \frac{1}{\pi} \int_{\S\setminus \g} \abs{\nabla \sigma_{h_T}(z)}^2 \dd z^2$$ 
is the Dirichlet energy of
$$\sigma_{h_T} (z) := \log \abs{h_T'(z)}.$$

   It is worthwhile noticing that this energy is the same for $h=h_T$ as for its inverse map $\varphi = h^{-1}$.  
   More precisely, one has
      $\sigma_h \circ \varphi  = - \sigma_{\varphi}  $ and 
     \begin{align} \label{eq_inverse_J}
     \begin{split}
     \frac{1}{\pi} \int_{\S}\abs{ \nabla \sigma_{\varphi}(z)}^2 \dd z^2 & = \frac{1}{\pi} \int_{\S}\abs{ \nabla (\sigma_h \circ \varphi (z))}^2 \dd z^2 \\
     & = \frac{1}{\pi} \int_{\S }\abs{ \nabla \sigma_h }^2(\varphi(z)) \abs{\varphi'(z)}^2 \dd z^2  \\
     &= \frac{1}{\pi} \int_{\S \setminus \g }\abs{ \nabla \sigma_h}^2 (y) \dd y^2.
     \end{split}
     \end{align}
     
We will first consider regular enough curves in the proof of Proposition~\ref{prop:eq_chord_1}, the following theorem is useful which states that the regularity of the curve is characterized by the regularity of its driving function:
recall that $C^{\a}$ is understood as the H\"older class $C^{k ,\b}$, where $k$ is the integer part of $\a$ and $\b = \a - k$, that are $C^{k}$ functions with  $\b$-H\"older continuous $k$-th derivative.
 \begin{theorem}[see \cite{LoopEnergy,carto-wong}] \label{thm_driving_regularity_eq} For $1<\a<2$, $\a \neq 3/2$,
A simple curve $\g$ is $C^{\a}$ if and only if  it is driven by a $C^{\a - 1/2}$ function. 
\end{theorem}
It allows us to deduce the regularity of the completed chord $\ad \g$ from the regularity of~$\g$.
\begin{cor}\label{cor_regularity_extension} 
  If $T < \infty$, $0< \a \leq 1$ and $\g [0,T]$ is $C^{1,\a}$. Then $\ad \g$ is $C^{1,\b}$, where $\b  = \a$ if $\a < 1/2$, and $\b$ can take any value less than $1/2$ if $\a \geq 1/2$.
\end{cor}
\begin{proof}
From Theorem~\ref{thm_driving_regularity_eq}, the driving function $W$ of $\g$ is in $C^{\a +1/2}$ if $\a \neq 1/2$. The completion $\ad \g$ of $\g$ by conformal geodesic  is driven by $\tilde W: t \mapsto W(\min (t, T))$ which is $C^{\min(\a + 1/2, 1)}$. 
It in turn implies that $\ad \g$ is in $C^{1,\b}$. 

If $\a = 1/2$, it suffices to replace $\a$ by $1/2 - \vare $ for small enough $\vare$. 
  \end{proof}

\section{Weak J-Additivity} \label{sec_weak_J}
Recall that $I$ satisfies the additivity property (\ref{IAdd}). 
The first step in our proof of the identity $J=I$ in Proposition~\ref{prop:eq_chord_1} will  be to show that $J$ satisfies the same additivity property  in the case of 
regular curves $\g$ (this is our Proposition~\ref{prop_add} which is the purpose of this section).
More precisely, in this section, we only consider the case when $\g \cup \m R_+$ is $C^{1,\a}$ for some $\a >0$. 
From Theorem~\ref{thm_driving_regularity_eq}, this is equivalent to that the extended driving function $\ad W : (-\infty, T] \to \m R$ of $W$, such that $\ad W (t) = 0$ for $t \leq 0$ has H\"older exponent strictly larger than $1/2$. 
In fact, $\ad W$ is the driving function for the embedded arc $\g \cup \m R_+$ rooted at $\infty$ (see Section~\ref{sec_loop} for more details on the extension of driving functions). 

Let us first recall some classical analytic tools:
Let $D$ be a Jordan domain with boundary $\G$ and let $\varphi$ be a conformal mapping from $\m D$ onto $D$. From Carath\'eodory theorem (see e.g. \cite{GM2005Harmonic} Thm.~I.3.1), $\varphi$ can be extended to a homeomorphism from $\overline {\m D}$ to $\overline{D}$. 
Moreover, the regularity of $\varphi$ is related to the regularity of $\G$ from Kellogg's theorem:
\begin{theorem}[Kellogg's theorem, see e.g. \cite{GM2005Harmonic} Thm.~II.4.3] \label{thm_Cna} Let $n \in \m N^*$, and $0< \a <1$. Then the following conditions are equivalent :
  \begin{enumerate}[(a)]
     \item \label{Cna_cond_1}$ \G$ is of class $C^{n,\a}$.
     \item \label{Cna_cond_2} $\arg (\varphi')$ is in $C^{n-1, \a} (\partial \m D)$.
     \item \label{Cna_cond_3} $\varphi \in C^{n,\a} (\ad {\m D})$ and $\varphi' \neq 0$ on $\ad{ \m D}$.
 \end{enumerate}
If one of the above condition holds, we say that $D$ is a $C^{n, \a}$ domain. 
When $\a = 0$, conditions \ref{Cna_cond_1} and \ref{Cna_cond_2} are still equivalent.
\end{theorem}

 An unbounded domain is said to be $C^{n,\a}$ if there exists a M\"obius map mapping it to a bounded $C^{n, \a}$ domain. 
 Now let $H$ be a $C^{1,\a}$ domain with $0< \a <1$ and $0,\infty \in \partial H$.  
 We parametrize its boundary $\G$ by arclength $\G : \m R \to \partial H$, such that $\G(0)= 0$. Let $\phi$ be a conformal map fixing $\infty$ from $H$ onto $\m H$.
Conjugating by a M\"obius transformation, Theorem~\ref{thm_Cna} implies that $\phi^{-1}$ is $C^{1,\a}$ in all compacts of $\overline {\m H}$. 
Since $(\phi^{-1})'$ is locally bounded away from $0$, the inverse function theorem shows that $\phi$ is also $C^{1,\a}$ in all compacts of $\overline {H}$. 
In particular, both $\s_{\phi} = \log \abs{\phi'}$ and its conjugate $\nu_{\phi} = \arg (\phi ') = \Im \log (\phi')$ are $C^{ \a}$ in all compacts of $\overline {H}$.

 \begin{lem}[Extension of Stokes' formula]\label{lem_bv}
 For a $C^{1,\a}$ domain $H$ as above and  all smooth and compactly supported functions $g \in C_c^{\infty}(\ad{H})$, 
 \begin{equation} \label{eq_curvature}
      \int_{H} \nabla g(z) \cdot \nabla \sigma_{\phi}(z) \dd z^2 = - \int_{\m R} g(\G(s)) \dd \tau(s),
 \end{equation}
 where $\tau(s) := \arg (\G'(s))$ is chosen to be continuous, and the right-hand side is a Riemann-Stieljes integral.
 \end{lem}
 
 The existence of the Riemann-Stieljes integral against $\dd \tau(s)$ is due to a classical result of Young \cite{young1936Holder}: 
 \begin{theorem}[Young's integral] \label{thm_Young}
 If $X \in C^{\a} ([0,T], \m R)$ and $Y \in C^{\b} ([0,T], \m R)$,  $\a +\b >1$, $\a, \b \leq 1$, then the limit below exists and we define
\[\int_0^T Y(u) \dd X(u) : = \lim_{\abs{P} \to 0} \sum_{(u,v) \in P } Y(u ) (X(v) - X(u)) \]
 where $P$ is a partition of $[0,T]$,  $\abs{P}$ the mesh size of $P$. 
The above limit is also equal to 
\[\lim_{\abs{P} \to 0} \sum_{(u,v) \in P } Y(v ) (X(v) - X(u)) \]
and the integration by parts holds:
\[\int_0^T Y(u) \dd X(u)  = Y(T) X(T) - Y(0)X(0)- \int_0^T X(u) \dd Y(u).\]
Moreover, one has the bounds:  for $0 \leq s<t \leq T$,
 \begin{enumerate}[(a)]
     \item $ \abs{\int_s^t Y(u) - Y(s) \dd X(u)} \lesssim \norm{Y}_{\b} \norm{X}_{\a} \abs{t-s}^{\a + \b} $.
     \item $\norm{ \int_0^{\cdot} Y(u) \dd X(u)}_{\a} \lesssim (\abs{Y(0)} + \norm{Y}_\b) \norm{X}_{\a}$,
 \end{enumerate} 
 where $\lesssim$ means inequality up to a multiplicative constant depending only on $\a, \b$ and $T$. 
 \end{theorem}

Notice that when $\G$ is smooth, the outer normal derivative $\partial_n \s_{\phi}$ is well defined on the boundary, the above lemma is indeed Stokes' formula
\begin{align*}
\int_{H} \nabla g(z) \cdot \nabla \sigma_{\phi}(z) \dd z^2 & =-\int_{H} g(z) \D \sigma_{\phi}(z) \dd z^2 + \int_{\G} g(z) \partial_n \s_{\phi}(z) \dd l(z) \\
&= \int_{\G} - g(z) k_0(z) \dd l(z) = \int_\G -g(z) \dd \tau, 
\end{align*}
where $k_0(z)$ the geodesic curvature of $\partial H$ at $z$ and $\dd l$ is integration with respect to the arclength on the boundary.
In this case, the first equality is due to the fact that $g$ is compactly supported. 
The second equality follows from the harmonicity of $\s_{\phi}(z)$ and Lemma~\ref{lem:geodesic_curvature} which gives
\begin{equation*} 
\partial_n \s_{\phi} (z) = k (\phi(z)) e^{\s_{\phi} (z)} - k_0(z),
\end{equation*}
where $k(\phi(z))$ is the geodesic curvature of $\partial \m H$ at $\phi(z)$ which is zero. 
Hence, the lemma's goal is to deal with the case where the boundary is less regular as the geodesic curvature is not defined for $C^{1,\a}$ domains.

\begin{proof}  [Lemma~\ref{lem_bv}]
Let  $H_{\vare} = \phi^{-1} (\m H +i\vare)$ be the domain with boundary $\G_{\vare} = \phi^{-1} (\m R + i\vare)$ parametrized by arclength: $s \to \G_{\vare} (s)$. We choose the parametrization such that $\G_{\vare}(0) \to \G(0)$ as $\vare \to 0$.
Since $\G_{\vare}$ is analytic, the remark above applies and one gets
\begin{align*}
\int_{H_{\vare}} \nabla g(z) \cdot \nabla \sigma_{\phi}(z) \dd z^2 &= \int_{\G_{\vare}}  g(z) \partial_n \s_{\phi}(z) \dd z = \int_{\m R}  g(\G_{\vare}(s)) \partial_{s} \nu_{\phi}(\G_{\vare}(s)) \dd s \\
&
= \int_{\m R}  - \partial_{s}  g(\G_{\vare}(s)) \nu_{\phi}(\G_{\vare}(s)) \dd s
\end{align*}
by integration by parts. Since $\phi^{-1}$ is $C^{1,\a}$ in all compacts of $\ad{\m H}$, the bijective map $\psi$ from $\ad{\m H}$ to itself $(x, y) \mapsto (s, y)$ such that $\G_{y} (s) = \phi^{-1} (x+iy)$ is continuous. The inverse of $\psi$ is continuous therefore uniformly continuous on compacts.
Since 
$\G_{\vare} (\cdot) = \phi^{-1} \circ \psi^{-1} (\cdot, \vare)$, we have that $\nu_{\phi}(\G_{\vare}(\cdot))$ converges uniformly  on compacts to $\nu_{\phi}(\G(\cdot))$.
The above integral converges 
as $\vare \to 0$ to
\[\int_{\m R}  - \partial_{s}  g(\G (s)) \nu_{\phi}(\G (s)) \dd s = \int_{\m R}  g(\G (s)) \dd \nu_{\phi}(\G (s)) = \int_{\m R}  - g(\G (s)) \dd \tau (s), \]
since $g(\G (\cdot))$ is at least $C^1$ and $\nu_{\phi}(\G (\cdot))$ is $C^\a$ in the support of $g(\G (\cdot))$, the integration by parts in the first equality holds. In the second equality, we use $\dd \nu_{\phi}(\G(s)) = - \dd\arg (\G'(s)) = -\dd \tau (s)$.
  \end{proof}

Now we would like to apply Lemma~\ref{lem_bv} to the special case of the slit domain $\S \setminus \g $ where $\g \cup \m R_+$ is at least $C^{1,\a}$, $\a >0$. 
A little bit of caution is needed because this is not a $C^{1,\a}$ domain.  However, Corollary~\ref{cor_regularity_extension} shows that the completion $\ad \g$ of $\g$ by conformal geodesic connecting $\g(T)$ and $\infty$ in $\S \setminus \g$ is $C^{1,\b}$ for some $0 < \b < 1/2$.  The complement of $\ad \g \cup \m R_+$ has two connected components $H_1$ and $H_2$, both are unbounded $C^{1,\b}$ domains. 
In fact, the regularity of $\ad \g \cup \m R_+$ at $\infty$ (after being mapped to a finite point via  M\"obius transformation) can be easily computed and is at least $C^{1, 1/2}$, see \cite[Prop.~3.12]{MR07zipper}.
And the mapping-out function~$h = h_T$ maps both domains to $\m H$ and the lower-half plane $\m H^*$ respectively. 

We parametrize $\G = \ad \g \cup \m R_+$ by arclength such that $\G(0) = 0$ and consider it as the boundary of $H_1$ (so that $H_1$ is on the left-hand side of $\G$), we denote by $\tilde \G (s) = \G(-s)$ the arclength-parametrized boundary of $H_2$ (see Figure~\ref{fig_thm_chord}).

For a domain $D \subset \m C$, we introduce the space of smooth functions with finite Dirichlet energy: 
$$\mc D^{\infty} (D) : = \left\{ g \in C^{\infty} (D), \, \int_D \abs{\nabla g (z) }^2\dd z^2 < \infty \right\}.$$

\begin{prop} \label{prop_zero} If a finite capacity curve $\g$ in $(\S, 0, \infty )$ satisfies:
\begin{itemize}
\item $\g \cup \m R_+$ is $C^{1,\a}$ for some $\a >0$, 
\item $\s_h$ is in $\mc D^{\infty}(\S \setminus \g)$.
\end{itemize}
Then for all $g \in \mc D^{\infty}(\S)$,
   \begin{equation*} 
   \int_{\S \setminus \g}\nabla g(z) \cdot \nabla \s_h (z) \dd z^2 = 0.
   \end{equation*}
\end{prop}

\begin{proof}
We have already seen that $H_1$ and $H_2$ are $C^{1,\b}$ domains for some $\b > 0$. 

Assume first that $g \in \mc D^{\infty} (\S)$ is compactly supported (in $\m C$) and that both $g|_{H_1}$ and $g|_{H_2}$ can be extended to $C^{\infty} (\ad{H_1})$ and $C^{\infty} (\ad{H_2})$ (with possibly different values along $\m R_+$), then Lemma~\ref{lem_bv} applies:
\begin{align*}
\int_{\S \setminus \g} \nabla g(z) \cdot \nabla \s_h(z) \dd z^2
&=  (\int_{H_1} + \int_{H_2})  \nabla g(z) \cdot \nabla \s_h(z) \dd z^2  \\
& = - \int_{\m R} g(\G(s)) \dd \tau(s) - \int_{\m R} g(\tilde \G(s)) \dd \tilde{\tau}(s) 
\end {align*} 
where $\tau (s) = \arg (\G'(s))$, and $\tilde \tau(s) = \arg(\tilde \G'(s)) = \tau(-s) + \pi$.

Since $\G(s) \in \S$ for $s < 0$, and $\dd \tau (s) = 0$ for $s \geq 0$, it follows that this quantity is also equal to 
 \begin{align*}
 &- \int_{-\infty}^0 g(\G(s)) \dd \tau(s) - \int_{0}^{+\infty} g(\G(-s)) \dd \tau(-s)\\
  = &- \int_{-\infty}^0 g(\G(s)) \dd \tau(s) - \int_{0}^{-\infty} g(\G(t)) \dd \tau(t) =  0.
 \end{align*}
The conclusion then follows from the density of compactly supported functions in $\mc D^{\infty}(\S)$ and the assumption $\s_h \in \mc D^{\infty}(\S \setminus \g)$.
  \end{proof}

 We are now ready to state and prove the $J$-additivity for sufficiently smooth curves: 
 Let $h_t$ be the mapping-out function of $\g[0,t]$ as in the proof of Proposition~\ref{prop_zero}. 
We write $h_{t,s} = h_t \circ h_s^{-1}$ for the mapping-out function of $h_s(\g[s,t])$,  where $s < t$. 

\begin{prop}[Weak J-Additivity]\label{prop_add} If $\g$ is a simple curve in $(\S, 0, \infty)$ such that  $\g \cup \m R_+$ is $C^{1,\a}$.
  For $0\leq s< t\leq T$, if both $J(h_s)$ and $J(h_{t,s})$ are finite, then 
$J(h_t) = J(h_s) + J (h_{t,s})$.
\end{prop}

\begin{proof}  Let $\g := \g [0,t]$, $\hat \g := h_s(\g[s,t])$.  We write $\s_r(z) = \log \abs{h_{r}' (z)}$ and $\sigma_{t,s} (z)= \log \abs{h_{t,s}' (z)}$.
From
\[\s_t(z) = \log \abs{h_t'(z)} = \log \abs{(h_{t,s} \circ h_s)'(z)} = \s_{t,s}(h_s(z)) + \s_s(z),\]
we deduce
\begin{align*}
\pi J(h_t) = & \, \pi J(h_s) + \int_{\S \setminus \g} \Big| \nabla \s_{t,s} (h_s (z)) \Big|^2 \dd z^2 \\
& + 2 \int_{\S \setminus \g}  \nabla \s_s(z)  \cdot \nabla \s_{t,s} (h_s (z)) \dd z^2.
\end{align*}

The second term on the right-hand side equals to $\pi J(h_{t,s})$ by the conformal invariance of the Dirichlet energy. 
Now we show that the third term vanishes. 
We write it in a slightly different way: it is equal to 
\begin{align*}
 \int_{\S \setminus \g}  - \nabla \s_{h_s^{-1}}(h_s(z)) \cdot \nabla \s_{t,s} (h_s (z)) \dd z^2 =   \int_{\S \setminus \hat \g}  - \nabla \s_{h_s^{-1}}(y)  \cdot \nabla \s_{t,s} (y) \dd y^2.
\end{align*}
The Dirichlet energy of $\s_{h_s^{-1}}$ is equal to $J(h_s)$.
Therefore from the assumption, $ \s_{h_s^{-1}} \in \mc D^{\infty} (\S)$ and $\s_{t,s} \in  \mc D^{\infty} (\S \setminus \hat \g)$. Since $\hat \g \cup \m R_+$ is at least $C^{1,\b}$ with the same $\b$ as in Corollary~\ref{cor_regularity_extension}, the vanishing follows from Proposition~\ref{prop_zero}. 
  \end{proof}

\section{The identity for piecewise linear driving functions} \label{sec_linear}
Let us first prove the identity between the Loewner energy of $\g$ and the Dirichlet energy of $\s_h$ in the special case of curve driven by a linear function:
Let  $\g$ be the Loewner curve in $(\S, 0, \infty)$ driven by the function $W: [0,T] \to \m R$, where $W(t) = \l t$ for some $\l \in \m R$. 
We denote by $(f_t : = g_t - W(t))_{t \in [0,T]}$ the centered Loewner flow in $\m H$ driven by $W$ and $(h_t)_{t \in [0,T]}$ the Loewner flow in $\S$. They are related by
$$h_t (z) = f_t^2(\sqrt z), \quad z \in \m C \setminus \m R_+.$$
In particular the mapping-out function $h$ is equal to $h_T$.  

We use the notations of $\G$ and $\tilde \G$ as in the description prior to Proposition~\ref{prop_zero} to distinguish the two copies of $\g \cup \m R_+$ as the boundary of $\S \setminus \g$. 
We also keep in mind that $\g$ is capacity parametrized and $\G$ is arclength-parametrized. 
We define $\tau(\G(s))$ and $\tau(\tilde \G(s))$ to be a continuous branch of $\arg (\G'(s))$ and $\arg (\tilde \G'(s))$.

\begin{prop} \label{prop_linear} 
Proposition~\ref{prop:eq_chord_1} holds when $\g$ is driven by a linear function. 
\end{prop}

First notice that the function $W(t) = \l t$ for $t\geq 0$ and $W(t) = 0$ for $t \leq 0$ is $C^{0,1}$. 
Therefore,  $\g \cup \m R_+$ is $C^{1,\a}$ for $\a <1/2$ by Theorem~\ref{thm_driving_regularity_eq}. Once we have shown that $J(h_{\vare}) < \infty$ for some $\vare > 0$, the weak $J$-additivity (Proposition~\ref{prop_add}) applies.
We can note that the $J$-additivity and the $I$-additivity imply that $J(h_T)$ and $I(\g[0,T])$ are both linear in $T$, so 
that it suffices to check that $I(\g[0,T]) \sim J(h_T)$ as $T \to 0$.

\begin{proof}
 Notice that $\g$ is in fact  a $C^{\infty}$ curve and it is only in the neighborhood of $0$ the regularity of $\g \cup \m R_+$ is $C^{1,\a}$. Hence, $\s_h$ is $C^{\infty}$ up to the boundary apart from $0$. 
 
 We first show that Stokes' formula applies and $J(h)$ equals to the boundary integral:
 \begin{equation}   \label{eq_stokes}
   J(h) = - \frac{1}{\pi} \int_{\G \sqcup \tilde \G } \sigma_{h}(z) \dd \tau(z) : = \lim_{\vare \to 0} - \frac{1}{\pi}  \int_{ \G \sqcup \tilde \G \setminus B(0, \vare) } \sigma_{h}(z) \dd \tau(z).
    \end{equation}
Since both $\tau$ and $\s_h$ are $C^{\infty}$ away from $0$, for a fixed $\vare >0$, the integral above is well-defined.

We need to be slightly more careful as the boundary of $\S \setminus \g$ is not regular enough at $\g_T$ and $0$  to apply Stokes' formula and $\s_h$ is not smooth up to the boundary to apply directly Lemma~\ref{lem_bv}.
The singularity at $\g_T$ is actually simple to deal with: We extend $\g$ to a $C^{\infty}$ curve $\ad \g$ going to $\infty$, since $\s_h$ is continuous across $\ad \g \setminus \g$, and $\dd \tau (z)$ has opposite sign on (the extended) $\G$ and $\tilde \G$, the sum of the integrals on both copies of $\ad \g \setminus \g$ cancels out.
It then suffices to check that the singularity at $0$ does not affect the application of Stokes' formula.

 For this, we use the Loewner flow to control the asymptotic behavior of $\nabla \s_h$ at $0$.
  The centered forward Loewner flow $f_t (\cdot):= g_t(\cdot) - W(t)$  of the simple curve $\sqrt \g$ in $\m H$ driven by $W(t) = \l t$ satisfies for all $z \in \S$,
     \begin{equation*}
       \partial_t f_t(z) = 2/f_t(z) - W'(t) = 2/f_t(z) - \l.
     \end{equation*}
     The mapping-out function $h_t= f_t^2(\sqrt z)$ for $\g[0,t]$ satisfies
\[\partial_t h_t(z) = 2 f_t(\sqrt z)(2/f_t(z) - \l) = 4 - 2 \l f_t(\sqrt z).\]
Taking the derivative in $z$,
\[h_t'(z) = f_t(\sqrt z) f_t'(\sqrt z)/\sqrt z \text { and } \partial_t h_t'(z) = -\l f_t'(\sqrt z)/\sqrt z. \]
We use the short-hand notation $\sigma_t$ for $\sigma_{h_t}$ and $\sigma_T$ for $\sigma_h$. We have
\begin{equation*}
\partial_t \sigma_t (z) = \Re \left(\partial_t h_t'(z)/h_t'(z)\right) =  - \l \Re( 1/ f_t(\sqrt z)).
\end{equation*}

Therefore for $z \in \m H$,
\begin{align*}
\s_t (z) & = - \l \Re \left (\int_0^t \frac{1}{f_r(\sqrt z)} \right) \dd r\\
& = \frac{-\l}{2} \Re \left(\int_0^t  \partial_r f_r(\sqrt z) + \partial_r W_r \dd r\right) \\
& = -\frac{\l}{2} \Big(\l t +  \Re (f_t(\sqrt z)) - \Re(\sqrt z)\Big).
\end{align*}

In particular as $z \to 0$,
 \[\abs{\nabla \s_T (z)}= \abs{\frac{\l}{2} \left ( \frac{f_T' (\sqrt z)}{2\sqrt z} - \frac{1}{2\sqrt z}\right)}
 =\abs{ \frac{\l}{2} \left ( \frac{h_T' (\sqrt z)}{2f_T (\sqrt z)} - \frac{1}{2\sqrt z}\right)}= O\left (\frac{1}{\abs{\sqrt z}}\right)\]
since $h'$ is bounded on the closure of the $C^{1,\a}$ domain and $f_T(\sqrt z)$ is bounded away from $0$ as $z \to 0$. 
It shows that $\norm{\nabla \s_T}_{L^2(B(0, \vare) )} \to 0$ and the integral of $\s_T \partial_n \s_T$ along a smooth arc of length $\vare$ inside $B(0, \vare) $ go to $0$ as $\vare \to 0$. Hence for every $\d >0$, there exists $\vare > 0$ and a sub-domain $\S_{\vare}$ of $\S \setminus \g$ with smooth boundary which coincides with  $\S \setminus \g$ outside of $B_{\vare}(0)$, such that $\vare \to 0$ when $\d \to 0$,
\[ \abs{J(h_T) - \frac{1}{\pi} \int _{\S_{\vare}} \abs{\nabla \s_T}^2 \dd z^2 }\leq \d,\]
and 
\[\frac{1}{\pi} \abs{\int_{\partial \S_{\vare}} \sigma_T(z) \partial_n \sigma_T(z) \dd l(z)-\int_{ \G \sqcup \tilde \G \setminus B(0, \vare) } \sigma_T(z) \partial_n \sigma_T(z) \dd l(z) } \leq \d.\]

It then suffices to apply Stokes' formula on $\S_{\vare}$. For this, we need to control the decay of $\nabla \s_T$ as $z \to \infty$: Taking the gradient of the expression of $\partial_t \s_t$, one gets:
\[\abs{\partial_t \nabla \sigma_t (z)} = \abs {\frac{\l f_t'(\sqrt z)}{ 2 f_t^2(\sqrt z) \sqrt z}} = O\left( \abs{z}^{-3/2}\right) \]
which implies
\begin{equation}\label{eq_sigma_bound}
\abs{\nabla \sigma_T (z)} = O(\abs{z}^{-3/2}).
\end{equation}
It allows us to apply Stokes' formula (one can look at the integral on the domain $\S_{\vare} \cap B(0,R)$ and see that the contribution of the contour integral on $\partial B(0,R)$ vanishes as $R \to \infty$), together with the harmonicity of $\s_T$, we have:
\[ \int_{\S_{\vare}} \abs{\nabla \s_T}^2 \dd z^2 = \int_{\partial \S_{\vare}} \sigma_T(z) \partial_n \sigma_T(z) \dd l(z),\]
which yields 
\[\abs{J(h_T) -\int_{ \G \sqcup \tilde \G \setminus B(0, \vare) } \sigma_T(z) \partial_n \sigma_T(z) \dd l(z) } \leq 2 \d. \]
Using $\partial_n \s_T(z) = - \partial_s \tau(z)$ on the smooth boundary of~$\S_{\vare}$, then let $\vare \to 0$ and $\d \to 0$, we obtain the identity \eqref{eq_stokes}.

\medbreak

Now we prove the identity 
\[I(\g) = - \frac{1}{\pi} \int_{ \G \sqcup \tilde \G} \sigma_{h}(z) \dd \tau(z).\]
Similar to the computation of $\s_t(z)$, $\nu_t (z) :=  \Im \log (h_t'(z))$ satisfies
\[\nu_t(z) = -\frac{\l}{2} \Im \int_0^t \partial_r f_r(\sqrt z) \dd r = -\frac{\l}{2} \Big( \Im (f_t(\sqrt z)) - \Im(\sqrt z)\Big).\]

Let $S$ be the total length of $\g [0,T]$. 
A point $\g_t$ on $\g$ can be considered as a point in both $\G$ and $\tilde \G$, and there is $s \geq 0$, such that $\g_t = \G(-s) = \tilde \G(s)$. We deduce from the expression of $\nu_t$, that for $0 \leq s \leq S$,
\[\tau (\G(-s)) = - \nu_T (\g_t) = -\frac{\l}{2} \Im (\sqrt {\g_t} ),\]
\[\dd\tau (\G(-s))  =  \frac{\l}{2} \Im (\partial _t \sqrt {\g_t} ) \dd t. \]
Similarly,
\[\tau (\tilde \G(s)) = - \nu_T (\g_t) +\pi = -\frac{\l}{2} \Im (\sqrt {\g_t} ) + \pi,\]
\[\dd\tau (\tilde \G(s))  =  -\frac{\l}{2} \Im (\partial _t \sqrt {\g_t} ) \dd t. \]
Hence the integral in \eqref{eq_stokes} equals to
\begin{align*}
J(h )  = &- \frac{1}{\pi} \int_{\g_t \in \G} \left(-\frac{\l}{2} \Re (f_T(\sqrt {\g_t}))\right)  \frac{\l}{2} \Im (\partial _t \sqrt {\g_t} ) \dd t \\
& - \frac{1}{\pi} \int_{\g_t \in \tilde \G} \left(-\frac{\l}{2} \Re (f_T(\sqrt \g_t))\right)  \frac{\l}{2} \Im (-\partial _t \sqrt {\g_t} ) \dd t  \\
= & \frac{\l^2}{4\pi} \int_0^T \Big(f_{T-t}(0^+) - f_{T-t}(0^-) \Big)  \Im \left( \partial_t \sqrt{\g_t}\right)\dd t.
 \end{align*}
The second equality holds because of the linearity of the driving function, and $s \mapsto f_s(0^+)>0$ and $s \mapsto f_s(0^-)<0$ are respectively the two Loewner flows starting from $0$. 
We also know that $\sqrt {\g_t}$ satisfies the backward Loewner equation, that is for $t \in (0,T]$,
\begin{equation}\label{eq:backward_flow}
\partial_t \sqrt {\g_t} = -2/\sqrt{\g_t} + \l.
\end{equation}

In fact, for a fixed $t \in [0,T]$, $\sqrt {\g_t}$ can be computed as follows. Consider the reversed driving function $\b^t: [0,t] \to \m R$ defined as
$$\b^t(s) := W(t) - W (t-s).$$
The reversed Loewner flow starting from $z \in \m H$ is the solution $[0,t] \to \m H$ to the differential equation: 
\begin{equation}\label{eq:backward_Loewner_flow}
 \partial_s Z^{t}_s(z) = - 2 / Z^{t}_s(z) + \dot \b^t(s) \quad \text{for } s \in [0,t], 
 \end{equation}
with the initial condition $Z^{t}_0(z) = z$ and we have $\lim_{y \downarrow 0} Z^{t}_t(iy) = \sqrt {\g_t}$, see \cite{BasicSLE}.
Since $\dot \b^t (s) \equiv \l$, we have 
 $Z^{t}_{s}( iy ) = Z^{T}_{s}(iy)$ for all $0 \le s \le t \le T$ and $y >0$. In particular, 
$$\sqrt{\g_t} = \lim_{ y \downarrow 0} Z^{t}_t(iy) =  \lim_{ y \downarrow 0} Z^{T}_t(iy).$$
For $t \in (0, T]$, the limit commutes with the differentiation in  \eqref{eq:backward_Loewner_flow}
which then gives \eqref{eq:backward_flow}. (See also \cite{STW2019} for the approach considering the singular differential equation \eqref{eq:backward_Loewner_flow} starting directly from $0$ when $W$ is sufficiently regular.)

From the explicit computation of the Loewner flow driven by a linear function in \cite{kkn2004}, we have the asymptotic expansions as $t \to 0$: 
$$f_t(0^+) = 2\sqrt t + O(t),\quad \sqrt {\g_t} = 2 i \sqrt t + O(t).$$
Hence as $T \to 0$,
\begin{align*}
 &\Big(f_{T-t}(0^+) - f_{T-t}(0^-) \Big)  \Im \left( \partial_t \sqrt{\g_t}\right)\\
 = & \Big(f_{T-t}(0^+) - f_{T-t}(0^-) \Big)  \Im \left( -2/ \sqrt{\g_t}\right)  \\
= &\frac{4 \sqrt {T-t}}{\sqrt t} (1+O(\sqrt T)),
\end{align*}
which yields
\begin{align*}
J(h_T) & = (1+O(\sqrt T))  \frac{\l^2}{\pi}\int_0^T  \sqrt{T-t} / \sqrt t \dd t  \\ 
&
= (1+O(\sqrt T))  \frac{\l^2 T}{\pi} \int_0^1  \sqrt{1-t} / \sqrt t \dd t    \\
&= \frac{\l^2}{2} (T +O(T^{3/2})).  
\end{align*}
By the weak $J$-additivity one gets $J(h_T) = \l^2 T /2 = I(\g [0,T])$. 
  \end{proof}

The weak $J$-additivity, the $I$-additivity and Proposition~\ref{prop_linear} do immediately imply the following fact: 
\begin{cor}\label{cor_piecewise_linear} 
Proposition~\ref{prop:eq_chord_1} holds when $\g$ is driven by a piecewise linear function.
\end{cor}

\section {Conclusion of the proof of Proposition~\ref{prop:eq_chord_1} by approximation} \label{sec_approx_finite_chord}

We now want to deduce Proposition~\ref{prop:eq_chord_1}, the result for general finite energy chords, from Corollary \ref {cor_piecewise_linear} by approximation. 
We give first the following lemma on the lower semi-continuity of $J(h)$ which is the key tool here: 

\begin{lem}\label{lem_lower_semi_J} If $T < \infty$, $(W^{(n)})_{n \geq 1}$ is a sequence of driving functions defined on $[0,T]$, that converges uniformly to $W$. Then
   \[J(h) \leq \liminf_{n \to \infty} J(h^{(n)}),\]
   where $h^{(n)}: = h^{(n)}_T$ is the Loewner flow generated by $W^{(n)}$ at time $T$, and $h$ is generated  by $W$. 
\end{lem}
\begin{proof}
   Let $\varphi = h^{-1}$ and $\varphi^{(n)} = (h^{(n)})^{-1}$. Since $W^{(n)}$ converges uniformly to $W$, $\varphi^{(n)}$ converges uniformly on compacts to $\varphi$. 
   We have also that 
   \begin{align*}
   \abs{\nabla \s_{\varphi^{(n)}} (z)}^2 =  \abs{\frac{\varphi^{(n)}(z)''}{\varphi^{(n)}(z)'}}^2 
   \end{align*}
   converges uniformly on compacts to $\abs{\nabla \s_{\varphi}(z)}^2$.
   Hence 
   \begin{align*}
 \liminf_{n\to \infty} J(\varphi^{(n)})  & =  \liminf_{n \to \infty} \sup_{K \subset \S} \frac{1}{\pi}\int_K \abs{\nabla \s_{\varphi^{(n)}}(z)}^2 \dd z^2 \\
   & \geq \sup_{K \subset \S} \liminf_{n \to \infty} \frac{1}{\pi}\int_K \abs{\nabla \s_{\varphi^{(n)}}(z)}^2 \dd z^2 \\
   &= \sup_{K \subset \S} \frac{1}{\pi}  \int_K \abs{\nabla \s_{\varphi}(z)}^2 \dd z^2 =  J(\varphi),
   \end{align*}
   where the supremum is taken over all compacts in $\S$. Then we conclude with \eqref{eq_inverse_J}.
  \end{proof}

We have the following corollary which gives the finiteness of $J$-energy when the Loewner energy is finite.
\begin{cor} \label{cor_IJ_comparison}
If $\g$ driven by $W: [0,T] \to \m R$ has finite Loewner energy in $(\S, 0, \infty)$, then $J(h) \leq I(\g)$. In particular, $\s_h \in \mc D^{\infty} (\S\setminus \g)$.
\end{cor}
\begin{proof}
    Take a sequence of piecewise linear functions $W^{(n)}$ such that $W^{(n)}$ converges to $W$ uniformly and 
    $$ I(W^{(n)} - W) = \frac{1}{2}\int_0^T \abs{W'^{(n)}(t)-W' (t) }^2 \dd t \xrightarrow[]{n \to \infty} 0.$$
   This is possible since the family of step functions is dense in $L^2 ([0,T])$. Thus we can find a sequence of step functions $Y_n$ which converges to $W'$ in $L^2$, and define $W^{(n)}(t) = \int_0^t Y_n(t) \dd t$. 
   The convergence is also uniform since
   \[\abs{W^{(n)}(t) - W(t)} \leq \int_0^t \abs{W'^{(n)}(s) - W'(s)} \dd s \leq  \sqrt T \sqrt{ 2 I(W^{(n)} - W)}\]
   by Cauchy-Schwarz inequality. Lemma~\ref{lem_lower_semi_J}  and Corollary~\ref{cor_piecewise_linear} imply that 
   \begin{align*}
  I(\g) =  I(W) = \lim_{n\to \infty} I(W^{(n)}) = \lim_{n\to \infty} J(h^{(n)})  \geq  J(h)
   \end{align*}
   as desired.
  \end{proof}

Given the finiteness of the $J$-energy, one can improve the $J$-additivity by dropping the regularity condition on $\g$. 
The following lemma is a stronger version of Proposition~\ref{prop_zero} by assuming only that $\g$ has finite Loewner energy. 
\begin{lem} \label{lem_strong_zero} 
  If $\g$  is a Loewner chain in $(\S, 0 , \infty)$ with finite Loewner energy and finite total capacity. Then for all $g \in \mc D^{\infty} (\S)$,
    \begin{equation}  \label{eq_zero_strong}
   \int_{ \S \setminus \g}\nabla g(z) \cdot \nabla \s_h (z) \dd z^2 = 0.
   \end{equation}
    
\end{lem}

\begin{proof}
 Take the same approximation of the driving function $W$ of $\g$ by a family of piecewise linear driving functions  $W^{(n)}$ as in the proof of Corollary~\ref{cor_IJ_comparison}. 
 Let $\g^{(n)}$ be the curve driven by $W^{(n)}$. Let $A = \sup_{n \geq 1} I(\g^{(n)}) \geq I(\g)$. We may assume that $A < \infty$.
Corollary~\ref{cor_IJ_comparison} implies that $J(h) \leq A$.
Moreover, from Corollary~2.2 in \cite{wang_loewner_energy}, every subsequence of $\g^{(n)}$ has a subsequence that converges uniformly to $\g$ as capacity-parametrized curves, thanks to the fact that they are $k$-quasiconformal curves with $k$ depending only on $A$. 
Hence, $\g^{(n)}$ converges uniformly to $\g$. 

Since $\g^{(n)}$ are all $C^{1,\a}$ for $\a < 1/2$, let $h^{(n)}$ be the mapping-out function of $\g^{(n)}$, one has 
  \[\int_{\S\setminus \g^{(n)}}\nabla g(z) \cdot \nabla \s_{h^{(n)}} (z) \dd z^2 = 0,\]
by Proposition~\ref{prop_zero}.
Since $ g$ and $\s_h$ are in $\mc D^{\infty}(\S \setminus \g)$,
  for every $\vare > 0$, there exists a compact set $K \subset \S \setminus \g$, such that 
   \[ \int_{(\S\setminus \g) \setminus K} \abs {\nabla g(z)}^2\dd z^2 \le \vare,\]
   which implies 
   \[\int_{(\S\setminus \g) \setminus K}\nabla g(z) \cdot \nabla \s_{h} (z) \dd z^2 \leq \sqrt { \pi A \vare}\]
   by Cauchy-Schwarz inequality. It also holds for $\s_{h^{(n)}}$.
As $\g^{(n)}$ converges uniformly to $\g$,  $\g^{(n)} \cap K = \emptyset$ for $n$ large enough, and $h^{(n)}$ converges uniformly to $h$ on $K$ (Carath\'eodory convergence \cite{duren1983} Thm.~3.1), we have 
\[\abs{\nabla \s_{h} (z) - \nabla \s_{h^{(n)}} (z)} = \abs{\frac{h''}{h'} (z) - \frac{(h^{(n)})''}{(h^{(n)})'}  (z)} \xrightarrow[ n\to \infty]{\text{unif. on } K} 0.\]
Hence,
\begin{align*} 
 & \abs{\int_{\S\setminus \g}\nabla g(z) \cdot \nabla \s_{h} (z) \dd z^2 } \\
 = &\abs{\int_{\S \setminus \g}\nabla g(z) \cdot \nabla \s_{h} (z) \dd z^2  - \int_{\S \setminus \g^{(n)}}\nabla g(z) \cdot \nabla \s_{h^{(n)}} (z) \dd z^2 } \\
 \leq &  \abs{\int_{K}\nabla g(z) \cdot \nabla \s_{h} (z) \dd z^2 - \int_{K}\nabla g(z) \cdot \nabla \s_{h^{(n)}} (z) \dd z^2 } +2\sqrt{\pi A\vare} \\
  & \xrightarrow[ n\to \infty]{} 2\sqrt {\pi A \vare}.
\end{align*}
Letting $\vare \to 0$, we get \eqref{eq_zero_strong}.
  \end{proof}

We use the same notation as in Proposition~\ref{prop_add} and
deduce the following strong $J$-additivity:

\begin{cor}[Strong $J$-additivity] \label{cor_strong_J_add}  If $\g$ has finite Loewner energy, then $J(h_t) = J(h_s) + J(h_{t,s})$ for $0 \leq s \leq t \leq T$.
\end{cor}

\begin{proof} Since $J(h_s) \leq \int_0^s W'(r)^2 /2 \dd r$ and $J(h_{t,s}) \leq \int_s^t W'(r)^2 /2 \dd r$ from Corollary~\ref{cor_IJ_comparison}, they are automatically finite when $I(\g) $ is finite. 
The proof then follows exactly the same line of thought as the weak $J$-additivity by applying Lemma~\ref{lem_strong_zero} with $g= \s_{h_s^{-1}}$.
   \end{proof}

Now we have all the ingredients for proving Proposition~\ref{prop:eq_chord_1}.

\begin{proof}
   
  Given Corollary~\ref{cor_IJ_comparison},  we only need to prove $J(h) \geq I(\g)$.
   
   Consider the following two functions 
   $$a(t) := J(h_t) \text{ and } b(t) := \frac{1}{2} \int_0^t W'(s) ^2 \dd s  = I(\g[0,t]).$$
   Both of them satisfy the respective additivity. 
    From the definition of absolute continuity, $b(\cdot)$ is AC on $[0,T]$. 
    By the additivity, Corollary~\ref{cor_IJ_comparison} and \ref{AC_interval},  $a(\cdot)$ is also an AC function. 
    Thus \ref{AC_derivative} implies that on a full Lebesgue measure set $S$, the functions $a(\cdot)$, $b(\cdot)$ and $W(\cdot)$ are differentiable and $b'(t) = W'(t)^2 /2$. Corollary~\ref{cor_IJ_comparison} shows in particular $a'(t) \leq b'(t)$. 
    Now it suffices to show that $b'(t) \leq a'(t)$ for $t \in S$.
    
    By additivity, without loss of generality, we assume that $t = 0$ and $T =1$. Consider $W^{(n)}$ obtained by concatenating $n$ copies of $W[0, 1/n]$, that is
    \[W^{(n)}(t) = \intpart{tn} W(1/n) + W(t - \intpart{tn}/n), \quad \forall t\in [0,1]. \]
    It is easy to see that $ I(W^{(n)})$ converges to $ I(W^{\infty})$, where $W^{\infty}$ is the linear function $t \mapsto t W'(0) $, since
    \begin{align*} 
 I(W^{(n)}) = n b(1/n) \xrightarrow[n \to \infty]{} b'(0) = W'(0)^2 /2 = I(W^{\infty}).
    \end{align*}
     We have also that $W^{(n)}$ converges uniformly to $W^{\infty}$. In fact, since $W$ is differentiable at $0$, for every $\vare >0$, there exists $n_0$, such that for all $n \geq n_0$, for all $t \leq 1/n$,
    \[\abs{W(t) - W'(0) t} \leq \vare/n.\]
    Hence for $t \in [0,1]$,
    \begin{align*}
    \abs{W^{(n)}(t)  - t W'(0)} &\leq \abs{W^{(n)}(\intpart{tn}/n)  - W'(0) \intpart{tn}/n }  + \abs{W(\d) - \d W'(0)} \\
    & = \intpart{tn} \abs{W^{(n)}(1/n)  - (1/n) W'(0)} + \abs{W(\d) - \d W'(0)} \\
    &\leq \vare (tn+1)/n \leq 2\vare,
    \end{align*}
where $\d = t - \intpart{tn}/n$.

The uniform convergence of driving function and Lemma~\ref{lem_lower_semi_J} imply that 
\[J(h^{\infty}) \leq \liminf_{n\to \infty} J(h^{(n)}) = \liminf_{n\to \infty} n a(1/n) = a'(0),\]
 where $h^{\infty}$ is the mapping-out function generated by $W^{\infty}$, $h^{(n)}$ is generated by $W^{(n)}$. The first equality follows from the $J$-additivity. 
From Proposition~\ref{prop_linear}, 
  \[J(h^{\infty}) = I(W^{\infty})= \abs{W'(0)}^2 /2 = b'(0) \]
  which yields $b'(0) \leq a'(0)$ and concludes the proof.
  \end{proof}

   \section{The Loop Loewner Energy} \label{sec_loop}
    
The generalization of the chordal Loewner energy to loops is first studied in \cite{LoopEnergy} and the goal in this section is to derive the loop energy identity Theorem~\ref{thm_loop_identity}.
Let $\g$ be a Jordan curve on the Riemann sphere $\Chat =  \m C \cup \{\infty\}$, that is parametrized by a continuous $1$-periodic function that is injective on $[0,1)$. The \emph{Loewner loop energy} of $\g$ rooted at $\g(0)$  is given by
    \[ I^L(\g, \g(0)) := \lim_{\vare \to 0} I_{\Chat \setminus \g[0, \vare], \g(\vare), \g(0)} (\g[\vare, 1]).\]
    We will use the abbreviation $I_{\g[0,\vare]}$ in the sequel for $I_{\Chat \setminus \g[0, \vare], \g(\vare), \g(0)}$.
    From the definition, the loop energy is conformally invariant (\ie invariant under M\"obius transformations): if $\mu : \Chat \to \Chat$ is a M\"obius function, then 
    \[I^L(\g, \g(0))  = I^L(\mu(\g), \mu(\g)(0)).\]
    Moreover, the loop energy vanishes only on circles (\cite[Sect.2.2]{LoopEnergy}).

The loop energy can be expressed in terms of the driving function as well: we first define the driving function of an embedded arc in $\Chat$ rooted at one tip of the arc. 
\emph{An embedded arc} is the image of an injective continuous function $\g :[0,1] \to \Chat$. 
We parametrize the arc by the capacity seen from $\g(0)$ as follows (and the capacity reparametrized arc is denoted as $t \mapsto \Gamma (t)$):
\begin{itemize}
  \item Choose first a point $\g(s_0)$ on $\g$, for some $s_0 \in (0,1]$. Define $\G (0)$ to be $\g(s_0)$. 
  \item Choose a uniformizing conformal mapping $\psi_{s_0}$ from the complement of $\g[0,s_0]$ onto $\m H$, such that $\psi_{s_0}(\g(s_0)) = 0$ and $\psi_{s_0}(\g(0)) = \infty$. 
  \item For $s \in (0,1]$, define the conformal mapping $\psi_s$ from the complement of $\g[0,s]$ onto $\m H$ to be the unique mapping such that the tip $\g(s)$ is sent to $0$, $\g(0)$ to $\infty$, and   $\psi_s \circ \psi_{s_0}^{-1} (z) = z + O(1)$
   as $z\to \infty$. 
   \item Set $\g(s) = \G(t)$ if the expansion of $\psi_s \circ \psi_{s_0}^{-1}$ at $\infty$ is actually
   \[\psi_s \circ \psi_{s_0}^{-1} (z) = z - W(t) + 2t/z + o(1/z),\]
   for some $W(t) \in \m R$ and $2t$ is called the \emph{capacity} of $\g[0,s]$ seen from $\g(0)$, relatively to $\g(s_0)$ and $\psi_{s_0}$. The capacity parametrization $s \mapsto t$ is increasing and has image $(-\infty,T]$ for some $T \in \m R_+$. We set $\G(-\infty) = \g(0)$. 
   \item We define $h_t := \psi_s^2$ to be the \emph{mapping-out function} of $\g[0,s]$, which maps the complement of $\g[0,s]$ to the complement of $\m R_+$ such that $h_t(\g(0)) = \infty$ and $h_t(\g(s)) = 0$. 
   \item The continuous function $W$ defined on $(-\infty, T]$ is called the \emph{driving function of the arc} $\g$.
   \item The \emph {Loewner arc energy} of $\g$ is the Dirichlet energy of $W$ which is 
   \[I^{A} (\g, \g(0)) = \int_{-\infty}^T W'(t)^2 /2 \dd t = \lim_{\vare \to 0} I_{\g[0,\vare]} (\g[\vare, 1]).\]  
   \end{itemize}

   Note that the capacity $t$, $h_t$ and $W(t)$ depend on the choice of $s_0$ and $\psi_{s_0}$. 
   A different choice  of $s_0$ and $\psi_{s_0}$ changes the driving function to 
   \begin{equation} \label{eq_equiv_driv}  \tilde W(t) = W (\l^2 (t+a))/ \l - W (\l^2 a)/ \l, \end{equation}
   for some $\l >0$ and $a\in \m R$. However, the Dirichlet energy of $W$ is invariant under such transformations. 
   
   \begin{figure}[ht]
 \centering
 \includegraphics[width=0.95\textwidth]{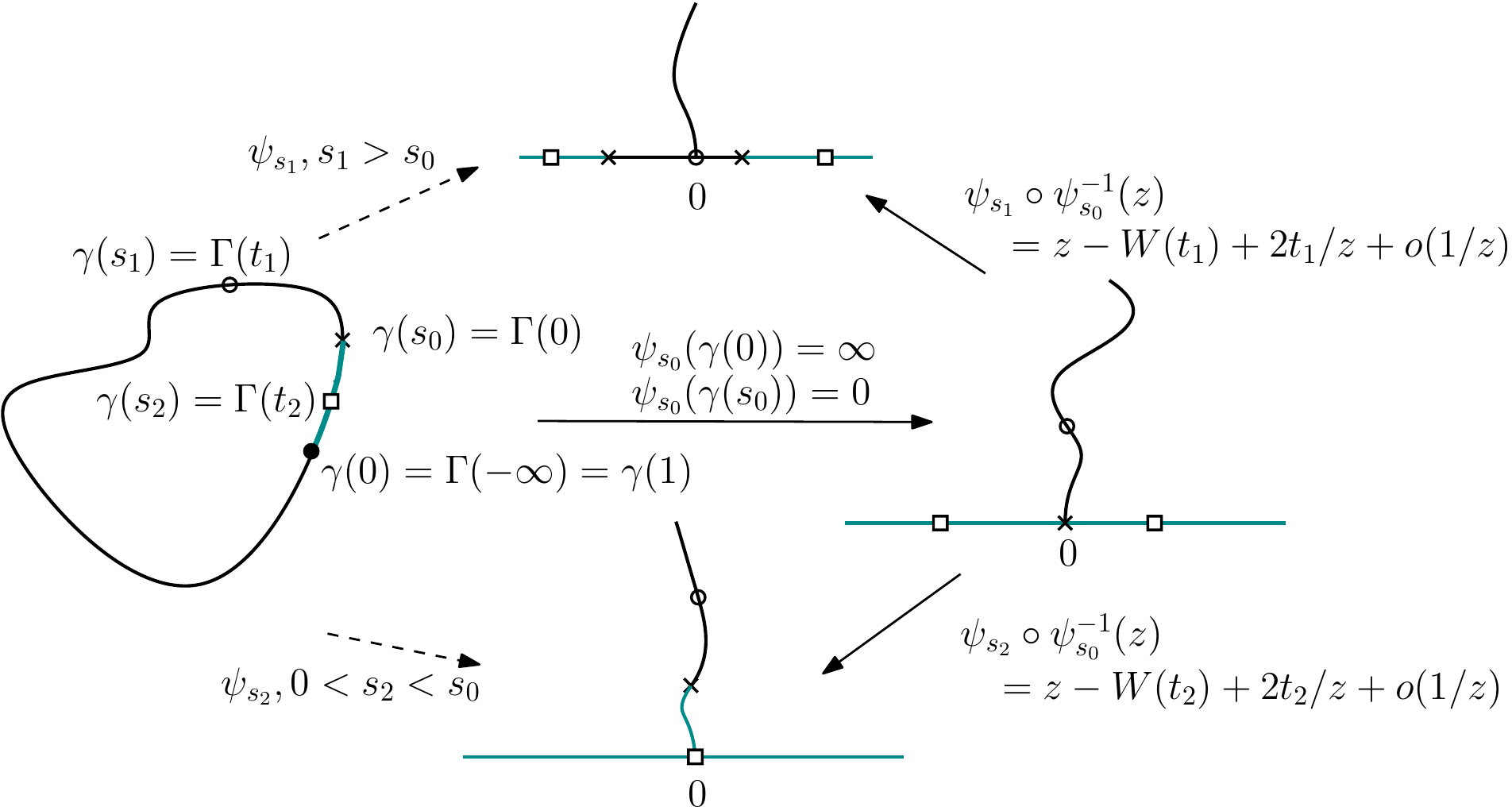}
 \caption{\label{fig_loop_driving} Illustration of the definition of loop driving function $W : \m R \to \m R$ and the capacity reparameterization $t\mapsto \G(t)$ of $\g$, where $0< s_2<s_0<s_1$ correspond to the capacities $-\infty < t_2 < 0 < t_1$.} 
 \end{figure}
   
    From the definition, as $T \to \infty$, the arc targets at its root to form a loop. This allows us to define the driving function of a simple loop $\g$ embedded in $\Chat$: 
    We parametrize and define the arc driving function of $\g[0,1-\vare]$ seen from $\g(0)$ for every $0<\vare < 1/2$. 
   With the same choice of $s_0$ and $\psi_{s_0}$, the capacity parametrization and the driving function of $\g[0,1-\vare]$ are consistent with respect to restrictions for all $\vare>0$. 
   Hence as $\vare \to 0$, $T \to \infty$ and we obtain the driving function $W :\m R \to \m R$ of the loop rooted at $\g(0)$.  
   Given the root $\g(0)$ and the orientation of the parametrization, the driving function is defined modulo transformations in \eqref{eq_equiv_driv}. 
   The loop energy is therefore the Dirichlet energy of the driving function $W$ which is invariant under those transformations.

It is clear that the loop energy depends \emph{a priori} on the root $\g(0)$ and the orientation of the parametrization, since the change of root/orientation induces non-trivial changes on the driving function that is in general hard to track. 
However, the main result of \cite{LoopEnergy} shows that the Loewner loop energy of $\g$ only depends on the image of $\g$. 
In this section we prove the identity (Theorem~\ref{thm_loop_identity}) that will give other approaches to the parametrization independence of the loop energy in Section~\ref{sec_detz} and \ref{sec_WP}.
Although we do not presume the root-invariance of the loop energy, we sometimes omit the root in the expression of Loewner loop energy. In this case, the root is taken to be $\g(0)$.

    From the conformal invariance of the Loewner energy, we may assume that $\g$ is a simple loop on $ \Chat$ such that $\g(0) = \infty$ and passes through $0$ and $1$. The complement of $\g$ has two unbounded connected components $H_1$ and $H_2$.
\begin{thm} \label{thm_loop_identity} If $\g$ has finite Loewner energy, then 
\[I^L(\g,\infty) = \frac{1}{\pi} \left ( \int_{\m C \setminus \g} \abs {\nabla{\s_{h}(z)}}^2 \dd z^2 \right),\]
  where $h|_{H_1}$ (resp. $h|_{H_2}$) maps $H_1$ (resp. $H_2$) conformally onto a half-plane and fixes $\infty$.
\end{thm}
  Notice that the expression $J(h)$ on the right-hand side does not depend on the orientation of the loop, but does \emph{a priori} depend on the special point $\infty$ which is the root of $\g$.
  
     We have mentioned in the introduction that the loop energy is a generalization of the chordal energy. 
      In fact, consider the loop $\g = \m R_+ \cup \eta$, where $\eta$ is a simple chord in $(\S, 0 , \infty)$ from $0$ to $\infty$, 
      and we choose $\g(0) = \infty$, $\g(s_0)= 0$, $\psi_{s_0}(\cdot)$ to be $\sqrt {\cdot}$, the orientation such that $\g[0,s_0] = \m R_+$. Then from the definition, the driving function of $\g$ coincides with the driving function of $\eta$ in $\m R_+$ and is $0$ in $\m R_-$.
      Hence 
      $$ I(\eta) = I^L(\eta \cup \m R_+ , \infty).$$
   Theorem~\ref {thm_chord_int_as_thm} follows immediately from Theorem~\ref{thm_loop_identity}.

\medbreak

As we described above, loops can be understood as embedded arcs with $T = +\infty$. For arcs which do not make it all the way back to its root ($T < \infty$), the mapping-out function $h_T$ is a natural choice for the uniformizing function $h$.
Let us first prove the analogous identity for an embedded arc.

\begin{lem} \label{lem_arc_identity} If $\g$ is a simple arc in $\Chat$ such that $\g(0 ) = \infty$ with finite arc energy. Then 
\[ J (h ) = I^A(\g, \infty), \]
where $h = h_T$ is a mapping-out function of $\g$.
\end{lem}

 \begin{figure}[ht]
 \centering
 \includegraphics[width=0.7\textwidth]{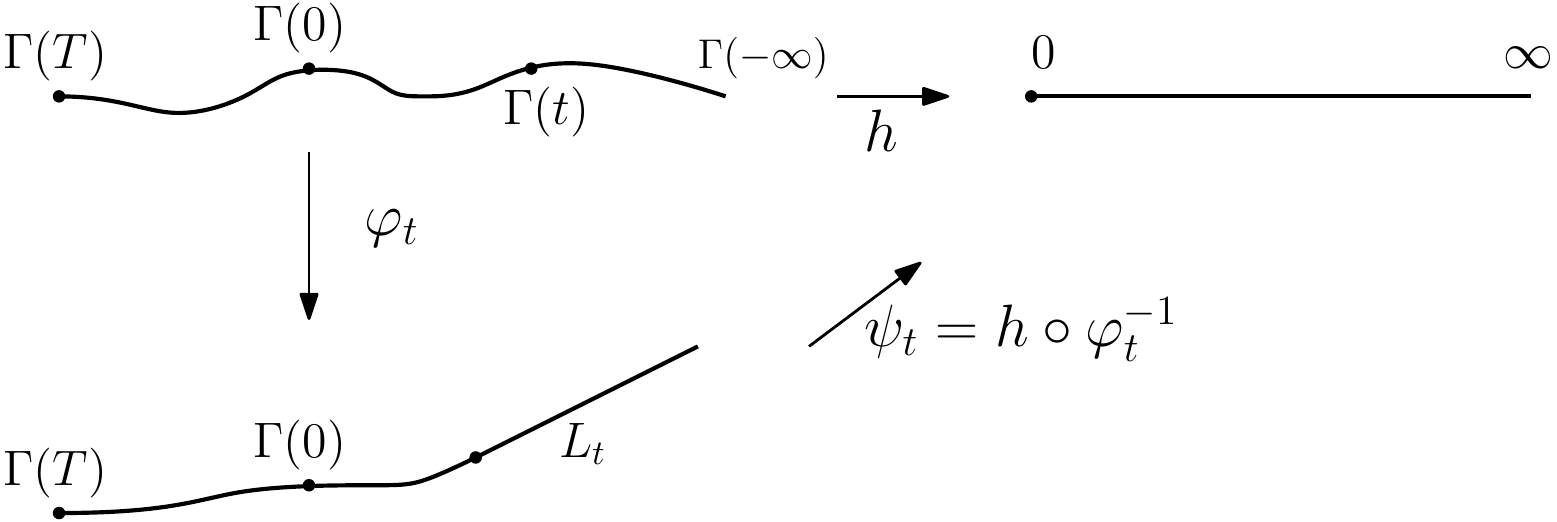}
 \caption{\label{fig_lem_arc} Conformal mappings in the proof of Lemma~\ref{lem_arc_identity} where $\varphi_t$ is defined in the complement of $\Gamma[-\infty,t]$ and $h$ in the complement of $\Gamma[-\infty ,T]$. Both of them map the tips to tips. } 
 \end{figure}
 
\begin{proof}
We will use the ``blowing-up at the root'' procedure to bring it back to the case of a finite capacity chord attached to $\m R_+$. Let $\G  [-\infty ,T] \to \Chat$ be the capacity reparametrization of $\g$ and $\G(-\infty) = \infty$  as described in the beginning of the section, and we choose  a point on $\g$ different from the tip $\g(1)$ to be $\g(s_0)$ so that $T >0$.   

   For every $t \in (-\infty,0]$, there exists a conformal mapping $\varphi_t$ fixing $\infty$, the tip $\G(T)$ and $\G(0)$  that maps the complement of $\G[-\infty, t]$ to a simply connected domain which is the complement of a half-line $L_t$. In fact, the mapping-out function of $\G[-\infty, t]$ maps the complement of $\G[-\infty,t]$ to the complement of $\m R_+$.
   Then we use a M\"obius transformation of $\Chat$ which sends the image of $\G(0)$ and $\G(T)$ back to $\G(0)$ and $\G(T)$ while fixing $\infty$. 

We prove first 
\begin{equation} \label{eq_ineq_arc} J(h) \leq I^A(\G[-\infty, T], \infty). \end{equation}
   For $n \in \m N$, the family $(\varphi_t|_{\Chat \setminus \G[-\infty, -n]})_{t \leq -n}$ is a normal family, and by diagonal extraction, there exists a subsequence that converges uniformly on compacts in $\m C$ to a conformal map $\varphi$ that can be continuously extended to $\Chat$. Since $\varphi$ fixes three points on $\Chat$, it is the identity map. 
   
     Let $\G^t$ be the curve which consists of the image of $\G[t,T]$ under $\varphi_t$ attached to the half-line $L_t$. 
     The map $\psi_t := h \circ \varphi_t^{-1}$ maps the complement of $\G^t$ to the complement of $\m R_+$, that fixes $\infty$. From Proposition~\ref{prop:eq_chord_1} and the invariance of $J$ under affine transformations, 
     \[J(h \circ \varphi_t^{-1}) = I_{L_t} (\varphi_t(\G[t,T])) = I_{\G[-\infty, t]} (\G[t, T]). \]
     Hence, it follows from the lower-semicontinuity of $J$ and the definition of arc Loewner energy that
     \[J(h) \leq \liminf_{t \to -\infty } J(h \circ \varphi_t^{-1}) =  I^A (\G[-\infty, T], \G(-\infty)).\]
     
     For the other inequality, it suffices to show that 
     \begin{equation}\label{eq:loop_proof_sum}
          J(h) = J(\varphi_t) + J(\psi_t)
     \end{equation}
    as it implies that 
     \[J(h) \geq J(\psi_t) = I_{\G[-\infty, t]} (\G[t, T]) \xrightarrow[t \to -\infty]{} I^A (\G[-\infty, T], \G(-\infty)).\]
     
     In fact, \eqref{eq:loop_proof_sum} is equivalent to 
     \[\int_{\Chat \setminus \G} \nabla \s_{\psi_t} (\varphi_t(z)) \cdot \nabla \s_{\varphi_t}(z) \dd z^2 = \int_{\Chat \setminus \G^t} - \nabla \s_{\psi_t} (y) \cdot \nabla \s_{\varphi_t^{-1}}(y) \dd y^2 = 0.\]
     Notice that $\varphi_t^{-1}$ is conformal in the complement of $L_t$. From \eqref{eq_ineq_arc}, $\s_{\varphi_t^{-1}} \in \mc D^{\infty} (\Chat \setminus L_t)$ and the curve attached to $L_t$ has finite chordal energy which is equal to $I_{\G[-\infty, t]} (\G[t, T])$.  Hence we conclude with Lemma~\ref{lem_strong_zero} by replacing $\m R_+$ by $L_t$.  
  \end{proof}

  The proof of Theorem~\ref{thm_loop_identity} then consists of making $T \to \infty$. The strategy is the same as the proof of Lemma~\ref{lem_arc_identity}.
  As we assume (without loss of generality) that $\g$ passes through $0$, $1$ and $\infty$, we can choose the uniformizing mappings $h|_{H_1}$ and $h|_{H_2}$ that fix $0$, $1$ and $\infty$ on the boundary.
 
   \begin {proof}[Proof of Theorem~\ref{thm_loop_identity}] 
    \begin{figure}
 \centering
 \includegraphics[width=0.9\textwidth]{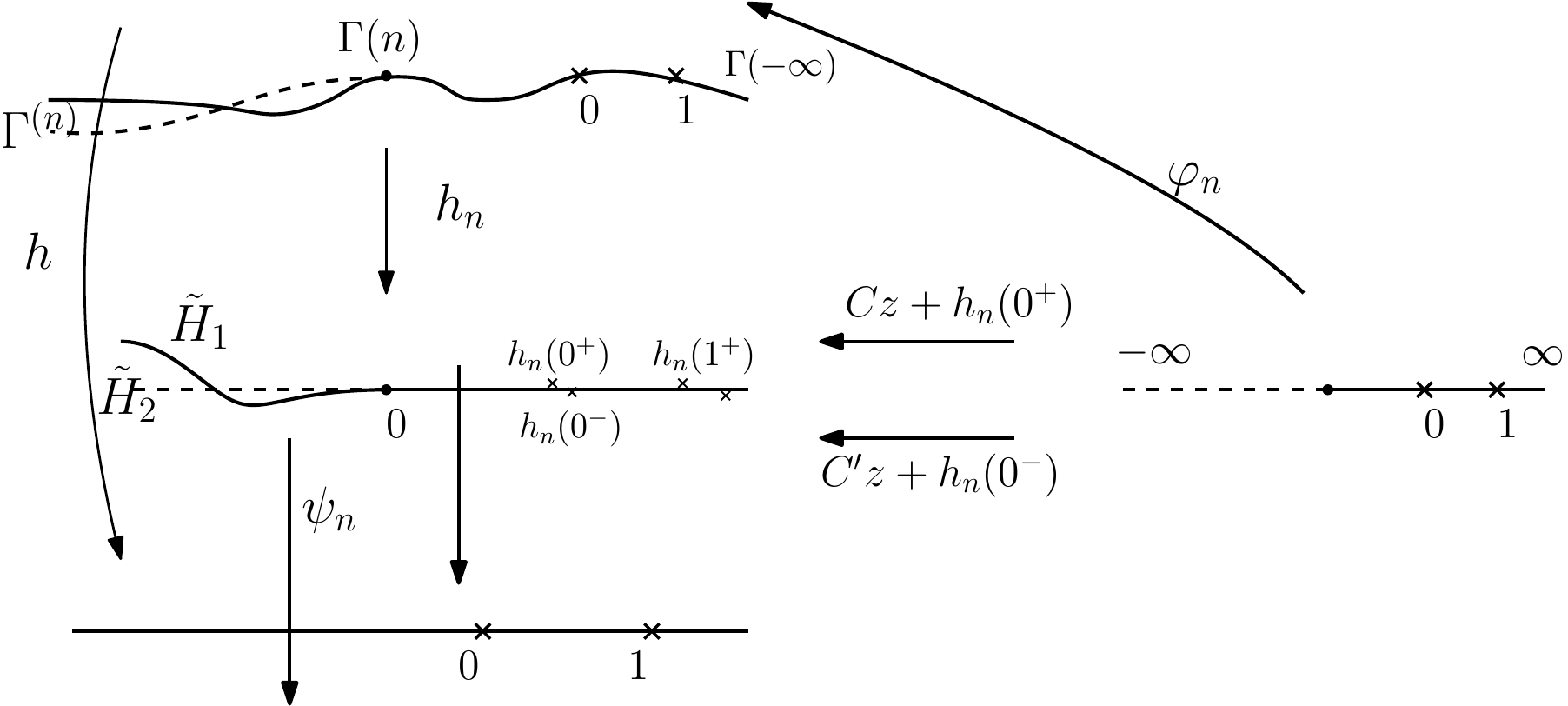}
 \caption{\label{fig_proof_thm_loop} Conformal mappings in the proof of Theorem~\ref{thm_loop_identity}. We define $\varphi_n (z) = (h_n)^{-1} (C z + h_n(0^+))$ on $\m H$ and $\varphi_n (z)= (h_n)^{-1} (C' z + h_n(0^-))$, where $C$ and $C'$ are chosen such that $\varphi_n$ fixes $0,1$ and $\infty$.} 
 \end{figure} 
   We prove first that $J(h) \leq I^L(\g, \g(0))$. Fix  a point $\g(s_0)$ on $\g$ and the conformal map $\psi_{s_0}$, let $(h_t)_{t \in \m R}$ be the mapping-out functions, $W$  the driving function and $\G$ the capacity reparametrized loop with $\G(-\infty) = \infty$.
   
     For $n \geq 0$, we consider $W^{(n)} (\cdot ) := W(\cdot \wedge n)$, and $\G^{(n)}$ the loop generated by $W^{(n)}$ which coincide with $\G$ on $[-\infty, n]$, that is the simple arc $\G[-\infty,n]$ followed by the conformal geodesic in $\m C \setminus \G[-\infty,n]$. 
   The mapping-out function $h_n$ of $\G[-\infty,n]$ maps both connected components $H^{(n)}_1$ and $H^{(n)}_2$ in the complement of $\G^{(n)}$ to half-planes.
   From Lemma~\ref{lem_arc_identity}, 
   \[I^L(\G^{(n)}) = I^A(\G[-\infty, n], \infty) = J(h_n).\]
   Notice that $h_n$ is not continuous on $\G[-\infty, n]$, we denote by $h_n(0^+)$ (resp. $h_n(0^-)$) the image of $0$ by $h_n|_{H_1}$ (resp. $h_n|_{H_2}$).
  Since $\G$ passes through $0$, $1$ and $\infty$ by assumption,  we define 
  $\varphi_n $ such that it maps respectively $\m H$ and $\m H^*$ to $H_1^{(n)}$ and $H_2^{(n)}$ while fixing $0$, $1$ and $\infty$. Let $\varphi = h^{-1}$.
   Since $(\varphi_n)_{n\geq 1}$ is a normal family, there exists a subsequence that converges uniformly on compacts, by Carath\'eodory kernel theorem, the limit is $\varphi$. Hence
   \[ I^L(\g) = \lim_{n\to \infty} J(h_n) =  \lim_{n \to \infty} J(\varphi_n) \geq J(\varphi). 
   \]

  Now we prove that $J(h) \geq I^L(\g)$. 
   Let $\psi_n := h \circ h_n^{-1}$ which maps each connected component $\tilde H_i := h_n(H_i)$ of $\S \setminus h_n( \G [n,\infty])$ to a half-plane, we have then
   \begin{equation}  \label{eq_loop_proof_cross}
   J(h) = J(\psi_n)+  J(h_n) + \frac{2}{\pi} \int_{\tilde H_1\cup \tilde H_2} \nabla \s_{h_n^{-1}} \cdot \nabla \s_{\psi_n} dz^2. 
   \end{equation}
   Lemma~\ref{lem_arc_identity} shows that $\s_{h_n^{-1}}$ has finite Dirichlet energy bounded by the arc Loewner energy of $\G[-\infty, n]$ hence by $I^L(\G, \infty)$. 
 On the other hand, the inequality $J(h) \leq I^L(\g)$ that we have proved above gives us the finiteness of the Dirichlet energy of $\s_{\psi_n}$: For every $\vare >0$, there exists $n_0$ large enough, such that $\forall n\geq n_0$,
   \[J(\psi_n) \leq I_{\m R_+} ( h_n(\G[n,\infty]) )= \int_n^{\infty} W'^2(t) /2 \dd t \leq \vare.\]

By the Cauchy-Schwarz inequality, the cross term in \eqref{eq_loop_proof_cross} converges to $0$ as $n \to \infty$, and $J(h_n)$ converges to $I^L(\G, \infty)$. 
Hence $J(h) \geq I^L(\g)$.
     \end{proof}
   
   \section{Zeta-regularized Determinants}\label{sec_detz}

In this section we prove the identity of the Loewner loop energy with a functional of zeta-regularized determinants of Laplacians (i.e., Theorem~\ref{thm_energy_determinant}
 which is the complete version of Theorem~\ref {thm_mr3}). This functional has also been studied in \cite{BFKM1994logdet} and is also reminiscent of the partition function formulation of the SLE/GFF coupling \cite{Dubedat_GFF}.

We first review the definition of zeta-regularized determinants of Laplacians \cite{RaySinger1971}: 
Let $\D$ be the Laplace-Beltrami operator with Dirichlet boundary condition on a compact surface $(D,g)$ with smooth boundary. In fact, all the statements below may hold under weaker regularity conditions. But for the well-definition of the zeta-regularized determinant, one needs (as far as we are aware) the boundary to be $C^{1,1}$ to get precise enough asymptotics of the trace of the heat kernel, and this condition is anyway much stronger than having finite Loewner energy boundary. 
Therefore, to stay on the safe side, we restrict ourselves in this section to smooth boundary domains to fit into the framework of \cite{BFK1992MV,BFKM1994logdet,OPS}.

The zeta-regularized determinant is defined, as its name indicates, through its zeta function:
$$\zeta_{-\D}(s) = \sum_{ j = 1}^{\infty} \l_j^{-s} = \frac{1}{\G(s)} \int_0^{\infty} t^{s-1} \text{Tr} (e^{t \D}) \dd t = \frac{1}{\G(s)} \int_0^{\infty} t^{s-1} \sum_{j=1}^{\infty} e^{-t \l_j} \dd t,$$
where $0 < \l_1 \leq \l_2 \cdots$ is the discrete spectrum of $-\D$.
From the Weyl's law \cite{weyl1912}, $\l_i$ grows linearly, $\zeta_{-\D}$ is therefore analytic in $\{\Re(s) > 1\}$. One extends $\zeta_{-\D}$ meromorphically to~$\m C$.

The trace of the Dirichlet heat kernel has an expansion as $ t \downarrow 0$ (see e.g. \cite{VanDenBerg1987} for $C^{1,1}$ domains):
$$\text{Tr} (e^{t\D}) = (4\pi t)^{-1} \left(\vol(D) - \frac{\sqrt{\pi t}}{2} l(\partial D) \right) + O(1),$$
where $l(\partial D)$ is the arclength of $\partial D$ and $\vol(D)$ the area of $D$ with respect to the metric $g$.
The zeta function is the Mellin transform divided by $\G(s)$ (\cite{BGV} Lemma 9.34) of the heat kernel, so that the above asymptotics imply that $\zeta_{-\D}$ has the following expansion near zero
$$\zeta_{-\D}(s) = O(s) + \lim_{t \searrow 0} \left[ \text{Tr} (e^{t\D}) - (4\pi t)^{-1} \left(\vol(D) - \frac{\sqrt{\pi t}}{2} l(\partial D) \right)\right];$$
it is therefore analytic in a neighborhood of $0$. 
The log of the \emph{zeta-regularized determinant of $-\D$} is defined as  
$$\log \detz (-\D): = - \zeta_{-\D}'(0).$$

The terminology ``determinant'' comes from the fact that 
$$-\zeta_{-\D}'(s) = \sum_{j= 1}^{\infty} \log (\l_j) \l_j^{-s},$$
so that if we take formally $s = 0$, we get
$$``-\zeta'_{-\D}(0) =  \log \left(\prod_{j= 1}^{\infty} \l_j \right) = \log \det (-\D)."$$

When $(M,g)$ is a compact surface without boundary, $\D$ has a one-dimensional kernel,  
and its regularized determinant $\detz'(-\D)$ is defined similarly by considering only the non-zero spectrum.

The zeta-regularized determinant of the Laplacian depends on both the conformal structure and the metric of the surface. 
Within a conformal class of metrics (two metrics $g$ and $g'$ are \emph{conformally equivalent} if $g'$ is a \emph{Weyl-scaling} of $g$, \ie $g' = e^{2\s} g$ for some $\s \in C^{\infty} (M)$), the variation of determinants is given by the so-called Polyakov-Alvarez conformal anomaly formula that we now recall (a proof of the formula can be found in \cite{OPS}).   

Let $(M,g_0)$ be a surface without boundary, and with the same notation for the metric, $(D,g_0)$ a compact surface with boundary. If $g = e^{2\s} g_0$ is a metric conformally equivalent to $g_0$, with the obvious notation associated to either $g_0$ or $g$, we denote by
\begin{itemize}
\item $\D_0$ and $\D_g$ the Laplace-Beltrami operator (with Dirichlet boundary condition for $D$), 
\item $\vol_0$ and $\vol_g$ the area measure, 
\item $l_0$ and $l_g$ the arclength measure on the boundary,
\item $K_0$ and $K_g$ the Gauss curvature in the bulk,
\item $k_0$ and $k_g$ the geodesic curvature on the boundary.
\end{itemize}
\begin{theorem} [Polyakov-Alvarez Conformal Anomaly Formula \cite{OPS}] 
\label{thm_polyakov}
For a compact surface $M$ without boundary, 
\begin{align*}
\log \detz' (-\D_{g}) = & -\frac{1}{6\pi} \left[ \frac{1}{2} \int_{M} \abs{\nabla_0 \s}^2 \dd \vol_0 + \int_M K_0 \s \dd \vol_0 \right] \\
 & + \log \vol_g (M)+ \log \detz' (-\D_0) - \log \vol_0(M).
\end{align*}
The analogue for a compact surface $D$ with smooth boundary is:
\begin{align*}
\log \detz (-\D_g) = & -\frac{1}{6\pi} \left[ \frac{1}{2} \int_{D} \abs{\nabla_0 \s}^2 \dd \vol_0 + \int_D K_0 \s \dd \vol_0 + \int_{\partial D} k_0 \s \dd l_0 \right] \\
 & - \frac{1}{4\pi} \int_{\partial D} \partial_n \s \dd l_0 + \log \detz (-\D_0),
\end{align*}
where $\partial_n$ is the outer normal derivative.
\end{theorem}

Let $M = S^2$ be the $2$-sphere equipped with a Riemannian metric $g$, $\g \subset S^2$ a smooth Jordan curve dividing $S^2$ into  two components $D_1$ and $D_2$. Denote by $\D_{D_i,g}$ the Laplacian with Dirichlet boundary condition on $(D_i, g)$.
We introduce the functional $\mc H (\cdot, g)$ on the space of smooth Jordan curves:

 \begin{align} \label{eq:def_H}
    \begin{split}
  \mc H (\g, g) : = & \log \detz'(-\D_{S^2, g}) -  \log \vol_g(S^2) \\
 & - \log \detz (-\D_{D_1, g}) -  \log \detz (-\D_{D_2,g}).      
    \end{split}
 \end{align}

As a side remark,  Burghelea, Friedlander and Kappeler \cite{BFK1992MV} (see also Lee \cite{Lee1997}) proved a Mayer-Vietoris type surgery formula for determinants of elliptic differential operators. 
In our case, it allows to express $\mc H$ by determinants of Neumann jump operators as in Theorem \ref{thm_mr3}. However, we will not use it in our proof.
  
\begin{theorem}[Mayer-Vietoris Surgery Formula \cite{BFK1992MV}] \label{thm_MV}
We have
$$\mc H (\g ,g) = \log \detz'(N(\g,g)) - \log l_g(\g),$$
where $N(\g, g)$ denotes the Neumann jump operator across the Jordan curve $\g$: for $f \in C^{\infty}(\g, \m R)$,
$$N(\g, g) f = \partial_{n_1} u_1 + \partial_{n_2} u_2,$$
where $n_i$ is the outer unit normal vector on the boundary of the domain $D_i$, $u_i$ is the harmonic extension of $f$ in $D_i$. 
\end{theorem}
Choosing the outer normal derivatives makes $N(\g, g)$ a non-negative, essentially self-adjoint operator. 
Its zeta-regularized determinant is defined similarly as for $-\D$: we use its positive spectrum to define the zeta function then take $-\log \detz N (\g, g)$ to be the derivative of zeta function's analytic continuation at $0$. 
Notice that the harmonic extensions $u_i$ depend on the metric only by its conformal class and the normal derivatives depend on the data of $g$ only in a neighborhood of $\g$.  
By simply applying the Polyakov-Alvarez formula, we obtain the following proposition.
  \begin{prop} \label{prop_weyl_invariant}
  The functional  $\mc H(\cdot, g)$ is invariant under Weyl-scalings.
  \end{prop}
  \begin{proof} Let $\s \in C^{\infty} (S^2, \m R)$ and $g = e^{2\s} g_0$, 
  \begin{align*}
  &\mc H(\g, g) - \mc H(\g, g_0)  \\
  = & \log \detz'(-\D_{S^2,g}) - \log \vol_g(S^2) - \left( \log \detz'(-\D_{S^2,0}) - \log \vol_0(S^2)\right)\\
  & - \sum_{i = 1}^2 \Big(\log \detz(-\D_{D_i, g}) - \log \detz(-\D_{D_i, 0})\Big)\\
  = & -\frac{1}{6\pi} \left[ \frac{1}{2} \int_{S^2} \abs{\nabla_0 \s}^2 \dd \vol_0 + \int_S^2 K_0 \s \dd \vol_0 \right] \\
  & - \sum_{i = 1}^2 \Big( -\frac{1}{6\pi} \left[ \frac{1}{2} \int_{D_i} \abs{\nabla_0 \s}^2 \dd \vol_0 + \int_{D_i} K_0 \s \dd \vol_0 + \int_{\partial D_i} k_{i,0} \s \dd l_0 \right] \\
  & \hspace{35pt}  - \frac{1}{4\pi} \int_{\partial D_i} \partial_{n_i} \s \dd l_0 \Big),
  \end{align*}
  where $k_{i, 0}$ is the geodesic curvature on the boundary of $D_i$ under the metric $g_0$.
  The domain integrals cancel out.
  And for $z \in \g$, we have $k_{1,0}(z) = - k_{2,0}(z)$, thus the terms $\int_{\partial D_i} k_{i,0} \s \dd l_0 $ sum up to $0$.
  We have also the relation (Lemma~\ref{lem:geodesic_curvature})
  $$\partial_{n_i} \s = k_{i, g} e^{\s}- k_{i, 0},$$
  which yields
  \begin{align*}
  \int_{\partial D_i} \partial_{n_i} \s \dd l_0 & = \int_{\partial D_i}k_{i, g} e^{\s} - k_{i, 0} \dd l_0 \\
  & =    \int_{\partial D_i}  k_{i, g} \dd l_g -  \int_{\partial D_i} k_{i, 0}  \dd l_0
  \end{align*}
  that sum up to zero as well.
    \end{proof}
  
  \begin{cor} \label{cor_conf_inv}$\mc H(\cdot, g)$ is conformally invariant: let $\mu$ be a conformal map from $S^2$ onto $S^2$, then
  $$\mc H(\g, g) = \mc H (\mu (\g), g).$$
  \end{cor}
  \begin{proof}
      We have
$$\mc H (\mu(\g), g) = \mc H(\g,  \mu^*g) = \mc H(\g,  g)$$
where $\mu^* g$ is the pull-back of $g$, that is conformally equivalent to $g$. The second equality follows from Proposition~\ref{prop_weyl_invariant}. 
    \end{proof}
  
 We are now ready to state the main result of this section:  
\begin{thm} \label{thm_energy_determinant}
If $g = e^{2\varphi} g_0$ is a metric conformally equivalent to the spherical metric $g_0$ on $S^2$, then:
\begin{enumerate}[(i)]
\item \label{item_circle_minimize} Circles minimize $\mc H(\cdot, g)$
 among all smooth Jordan curves.
\item \label{item_energy_determinant}  Let $\g$ be a smooth Jordan curve on $S^2$. 
We have the identity
\begin{align*}
I^L(\g, \g(0))& = 12 \mc H(\g, g) - 12 \mc H(S^1, g) \\
              & = 12 \log \frac{\detz(-\D_{\m D_1, g})\detz(-\D_{\m D_2, g})}{\detz(-\D_{D_1, g}) \detz(-\D_{D_2, g})},
\end{align*}
where $\m D_1$ and $\m D_2$ are the two connected components of the complement of $S^1$. 
\end{enumerate}
\end{thm}
Let us make two remarks: 
\begin{itemize}
\item The right-hand side in \ref{item_energy_determinant} does not depend on the root, so that the root-invariance of the loop energy for smooth loops follows.
\item We also recognize the functional introduced in \cite{BFKM1994logdet}, where they defined
$$h_g(\g) : = \log \detz(-\D_{D_1, g}) + \log \detz(-\D_{D_2, g}),$$
so that our identity above can be expressed as 
$$I^L(\g) = 12 h_{g} (S^1) - 12 h_g(\g).$$
\end{itemize}

\begin{proof} The second equality in \ref{item_energy_determinant} follows directly from the definition. Since $I^L(\g)$ is non-negative, \ref{item_energy_determinant} implies that $S^1$ minimizes $\mc H(\cdot, g)$.
Corollary~\ref{cor_conf_inv} implies that $\mc H(C, g)  = \mc H (S^1,g)$ for any circle $C$ and we get \ref{item_circle_minimize}.

Therefore it suffices to prove the first equality in \ref{item_energy_determinant} for $g = g_0$ by Proposition~\ref{prop_weyl_invariant}. We also assume that $S^1$ is a geodesic circle and both $\g$ and $S^1$ pass through a point $\infty \in S^2$. 
We use the stereographic projection $S^2\setminus\{\infty\} \to \m C$ from $\infty$ and the image of $D_1$, $D_2$, $\m D_1$ and $\m D_2$ are $H_1$, $H_2$, $\m H$ and $\m H^*$.
With a slight abuse we use the same notation for the induced metric in $\m C$:
$$g_0 (z) = \frac{4 \dd z^2}{(1+\abs{z}^2)^2} = : e^{2\psi(z)} \dd z^2,$$
and $\inprod{\cdot}{\cdot}_0 : = g_0(\cdot, \cdot)$.  
Let $h$ be a conformal map that maps respectively from $H_1$ and $H_2$ to $\m H$ and $\m H^*$ fixing $\infty$ as in previous sections and we put $f = h^{-1}$. 
Let $g_1$ be the pull-back of $g_0$ by $f$:
\begin{align*}
g_1 (z) =& f^* g_0 (z) = e^{2\psi (f(z))} \abs{f'(z)}^2 \dd z^2 \\
 =&  e^{2\psi(f(z)) - 2 \psi(z) + 2\s_f(z)} g_0 (z) := e^{2\s(z)} g_0(z),
\end{align*}
where $\s_f (z)= \log \abs{f'(z)} $ and we set 
$$\t (z) = \psi(f(z)) - \psi(z) $$ so that 
$$\s (z ) =\t (z) + \s_f(z). $$
From the Polyakov-Alvarez conformal anomaly formula:
\begin{align*} 
 &\log \detz(-\D_{H_1, g_0}) - \log \detz(-\D_{\m H, g_0})
= \log \detz(-\D_{\m H, g_1}) - \log \detz(-\D_{\m H, g_0}) \\
= & -\frac{1}{6\pi} \left[ \frac{1}{2} \int_{\m H} \abs{\nabla_0 \s}^2 \dd \vol_0 + \int_{\m H} K_0 \s \dd \vol_0 + \int_{\m R} k_0 \s \dd l_0 \right] - \frac{1}{4\pi} \int_{\m R} \partial_{n_0} \s \dd l_0.
\end{align*}
As in the proof of Proposition~\ref{prop_weyl_invariant}, the last term above cancels out when we sum up both variations in $\m H$ and $\m H^*$. 
We have $K_0 \equiv 1$, $k_0 \equiv 0$, but as we will reuse the proof in Section~\ref{sec_WP}, we keep first $K_0$ and $k_0$ in the expressions.
The right-hand side in \ref{item_energy_determinant} is equal to
\begin{align}\label{eq_three_terms}
\begin{split}
&\frac{1}{\pi} \int_{\m H \cup \m H^*} \abs{\nabla_0 (\s_f + \t)}^2 + 2K_0 \s_f +2 K_0 \t \dd \vol_0  +  \frac{2}{\pi} \int_{\m R } k_0 \s \dd l_0\\
= & \frac{1}{\pi} \int \abs{\nabla_0 \s_f}^2 \dd \vol_0 + \frac{2}{\pi} \int \left(\inprod{\nabla_0 \s_f}{\nabla_0 \t}_0 +K_0 \s_f \right) \dd \vol_0\\
 & + \frac{1}{\pi}\int \left(\abs{\nabla_0 \t}^2  +2K_0 \t \right)\dd \vol_0  +  \frac{2}{\pi} \int_{\m R } k_0 \s \dd l_0.
\end{split}
\end{align}
Since the Dirichlet energy is invariant under Weyl-scalings of the metric, the first term on the right-hand side of the equality is equal to $J(f)$, which is also $I^L(\g, \infty)$ by Theorem~\ref{thm_loop_identity}. As $k_0 \equiv 0$, we only need to prove that the sum of the second and the third terms vanishes.

We denote the quantities/operators/measures with respect to the Euclidean metric in $\m C$ without subscript, 
then we have
\begin{align*} 
  \D_0 & = e^{-2\psi} \D; \hspace{50pt} \partial_{n_0}  = e^{-\psi} \partial_n; \\
  \dd \vol_0 & = e^{2\psi } \dd z^2;  \hspace{50pt}  \dd l_0 = e^{\psi} \dd l;\\
  \partial_n \s_f (z) &= k(f (z)) e^{\s_f (z)} -k (z); \\
  \D_0 \psi & =  e^{-2\psi} \D \psi = e^{-2\psi} (K - e^{2\psi} K_0) = -K_0;\\
  \partial_{n_0} \psi &= e^{-\psi} \partial_n \psi =  e^{-\psi} (e^{\psi} k_0- k) =  k_0 -e^{-\psi} k.
\end{align*}
For the second term in \eqref{eq_three_terms}, from  Stokes' formula:
   \begin{align*}
    & \int_{\m H} \inprod{\nabla_0 \s_f}{\nabla_0 (\psi \circ f )}_0 \dd \vol_0\\
   & = \int_{\m R} \psi(f)  \partial_{n_0} \s_f \dd l_0 - \int_{\m H} \psi(f) \D_0 \s_f \dd \vol_0 =  \int_{\m R} \psi(f)  \partial_{n} \s_f \dd l  \\
   &= \int_{\m R}k(f) e^{\s_f} \psi(f)\dd l - \int_{\m R} k \psi (f) \dd l = \int_{\g} k \psi \dd l(z) - \int_{\m R} k \psi(f)\dd l,
   \end{align*}
   the contributions from the first term in the above expression cancels out when we sum up both sides. Similarly we have
    \begin{align*}
   \int_{\m H} \inprod{\nabla_0 \s_f}{\nabla_0 \psi}_0 \dd \vol_0 & = \int_{\m R} \s_f  \partial_{n_0} \psi \dd l_0 - \int_{\m H} \s_f \D_0  \psi \dd \vol_0 \\
   & = \int_{\m R} \s_f  \partial_{n_0} \psi \dd l_0  + \int_{\m H} K_0 \s_f \dd \vol_0.
   \end{align*}
   Hence the second term in \eqref{eq_three_terms} equals to
   $$ - \frac{2}{\pi} \int_{\m R \sqcup \m R}  \left ( k \psi (f) +\s_f  \partial_{n} \psi \right) \dd l .$$
   For the third term in \eqref{eq_three_terms}, notice that
  \begin{align*}
  \int_{\m H} \inprod{\nabla_0 (\psi \circ f)}{\nabla_0 \psi}_0 \dd \vol_0 & = \int_{\m R}\psi (f) \partial_{n_0} \psi \dd l_0 - \int_{\m H} \psi (f) \D_0  \psi  \dd \vol_0 \\
  &=\int_{\m R} \psi (f)  \partial_{n} \psi \dd l + \int_{\m H} \psi (f) K_0 \dd \vol_0.  
  \end{align*}
  Similarly,
   \begin{align*} & \int_{\m H} \inprod{\nabla_0 (\psi \circ f)}{\nabla_0 (\psi \circ f)}_0 \dd \vol_0 \\
   & = \int_{\m H} \inprod{\nabla_0 \psi}{\nabla_0 \psi}_0 \dd \vol_0= \int_{\m R} \psi \partial_{n} \psi \dd l + \int_{\m H} \psi K_0 \dd \vol_0.
   \end{align*}
   Hence the third term equals to
   \begin{align*}
   &\frac{1}{\pi} \int_{\m H \cup \m H^*} \inprod{\nabla_0 \t}{\nabla_0 \t}_0  +2 K_0 \t \dd \vol_0 \\
   &= \frac{2}{\pi} \left(\int_{\m R \sqcup \m R} \psi \partial_{n} \psi \dd l - \psi(f)  \partial_{n} \psi \dd l\right) = - \frac{2}{\pi} \int_{\m R \sqcup \m R} \t \partial_n \psi \dd l.
   \end{align*}
  Therefore the sum of the second and the third terms of \eqref{eq_three_terms} is equal to
   \begin{equation} \label{eq_vanish_det}
  \frac{2}{\pi} \int_{\m R \sqcup \m R}  -k \psi(f) - \s \partial_n \psi \dd l,
   \end{equation}
   which vanishes since $k, k_0 \equiv 0$ on $\m R$ and  $\partial_n \psi = e^{\psi} k_0 - k \equiv 0$ as well.
  \end{proof}

\section{Weil-Petersson class of loops} \label{sec_WP}
In this section we establish the equivalence between finite energy curves and Weil-Petersson quasicircles (we will prove Theorem~\ref{thm_energy_liouville}, which 
is the precise version of Theorem~\ref{thm_mr4}).

Let us start with some background material on the universal Teichm\"uller space $T(1)$ and the the Weil-Petersson Teichm\"uller space $T_0(1)$. We follow here the notations of \cite{TT2006WP}. 
We define
$$\m D =  \{z \in \m C , \abs{z} <1 \}, \quad \m D^* = \{z \in \m C, \abs{z} > 1\},$$
and let  $S^1 = \partial \m D$ be the unit circle.
Let QS$(S^1)$ be the group of sense-preserving quasisymmetric homeomorphisms of the unit circle (see e.g. \cite{lehto1973quasiconformal}), $\text {M\"ob}(S^1) \simeq \text{PSL}(2, \m R)$ the group of M\"obius transformations of $S^1$ and $\text{Rot}(S^1)$ the rotation group of $S^1$. 
The \emph{universal Teichm\"uller space} is defined as the right cosets
$$T(1) : = \text{M\"ob}(S^1) \backslash \QS(S^1) \simeq \{\varphi \in \QS(S^1), \,\varphi \text{ fixes } -1, -i \text{ and }1\}.$$
We write $[\varphi]$ for the class of $\varphi$.
From the Beurling-Ahlfors extension theorem, for every $\varphi \in \QS(S^1)$ fixing $-1, -i$ and $1$, there exists a unique $\a \in \text{M\"ob}(S^1)$ such that $\a (1) = 1$, and 
conformal maps $f$ and $g$ on $\m D$ and $\m D^*$ satisfying:
\begin{enumerate} [CW1.]
\item \label{CW_ext} $f$ and $g$ admit quasiconformal extensions to $\m C$.  
\item \label{CW_weld} $ \a \circ \varphi = g^{-1} \circ f|_{S^1}$.
\item \label{CW_f} $f(0) = 0$, $f'(0) = 1$, $f''(0 ) = 0$.
\item \label{CW_g_infty}$ g(\infty) = \infty$.
\end{enumerate} 
The conformal map $f$ admits a quasiconformal extension to $\m C$, means that the complex dilatation $\mu$ in $\m D^*$  of the extension, defined by
$$\mu_f (z) := \partial_{\ad z} f /\partial_{z} f (z), $$
is essentially uniformly bounded by some constant $k < 1$. 
Let $\mc U$ denote the set of conformal maps (univalent functions) on $\m D$, we have then
$$T(1) \simeq \{f \in \mc U, f(0)= 0, f'(0) = 1, f''(0) = 0, f \text{ admits q.c. extension to } \m C\}.$$
We say that $(f,g)$ are \emph{canonical conformal mappings associated to} $[\varphi] \in T(1)$.

Takhtajan and Teo have proved that $T(1)$ carries a natural structure of complex Hilbert manifold and that the connected component of the identity $T_0(1)$ is characterized by:
\begin{theorem} [\cite{TT2006WP} Theorem~2.1.12] \label{thm_TT_equiv_T01} A point $[\varphi]$ is in $T_0(1)$ if the associated canonical conformal maps $f$ and $g$ satisfy one of the following equivalent conditions:
\begin{enumerate}[(i)]
\item $\int_{\m D} \abs{f''(z)/f'(z) }^2  \dd z^2 < \infty;$
\item $\int_{\m D^*} \abs{g''(z)/g'(z) }^2  \dd z^2 < \infty;$
\item $\int_{\m D} \abs{S (f)}^2 \rho^{-1} (z) \dd z^2 < \infty;$
\item $\int_{\m D^*} \abs{S (g)}^2 \rho^{-1} (z) \dd z^2 < \infty,$
\end{enumerate}
where $\rho(z) \dd z^2= 1/(1-\abs{z}^2)^2 \dd z^2$ is the hyperbolic metric on $\m D$ or $\m D^*$ and 
$$S(f) = \frac{f'''}{f'} - \frac{3}{2} \left(\frac{f''}{f'}\right)^2$$
is the Schwarzian derivative of $f$.
\end{theorem}
\begin{theorem} [\cite{TT2006WP} Theorem~2.4.1]
The universal Liouville action ${\bf S_1}: T_0(1) \to \m R$ defined by 
\begin{equation}\label{eq_def_S1}
{\bf S_1} ([\varphi]) : = \int_{\m D} \abs{\frac{f''}{f'}(z)}^2 \dd z^2 + \int_{\m D^*}\abs{\frac{g''}{g'}(z)}^2 \dd z^2 - 4\pi \log \abs{g'(\infty)},
\end{equation}
where $g'(\infty) = \lim_{z\to \infty} g'(z) = \tilde g'(0)^{-1}$ and $\tilde g(z) = 1/g(1/z)$, is a K\"ahler potential for the Weil-Petersson metric on $T_0(1)$.
\end{theorem}

Notice that from Theorem~\ref{thm_TT_equiv_T01}, the right-hand side in \eqref{eq_def_S1} is finite if and only if $[\varphi] \in T_0(1)$.

We define similarly the \emph{universal Liouville action for quasicircles}. 
If $\g $ is a bounded quasicircle, we denote (and in the sequel) the bounded connected component of $\m C \setminus \g$ by $D$, and the unbounded connected component by $D^*$. 
Let $f$ be any conformal map from $\m D$ onto $D$, and $g$ from $\m D^*$ onto $D^*$ fixing $\infty$. 
Conformal maps from $\m D$ onto a quasidisk always admit a quasiconformal extension to $\m C$. We denote again by $f$ and $g$ their quasiconformal extension.
\begin{figure}[ht]
 \centering
 \includegraphics[width=0.6\textwidth]{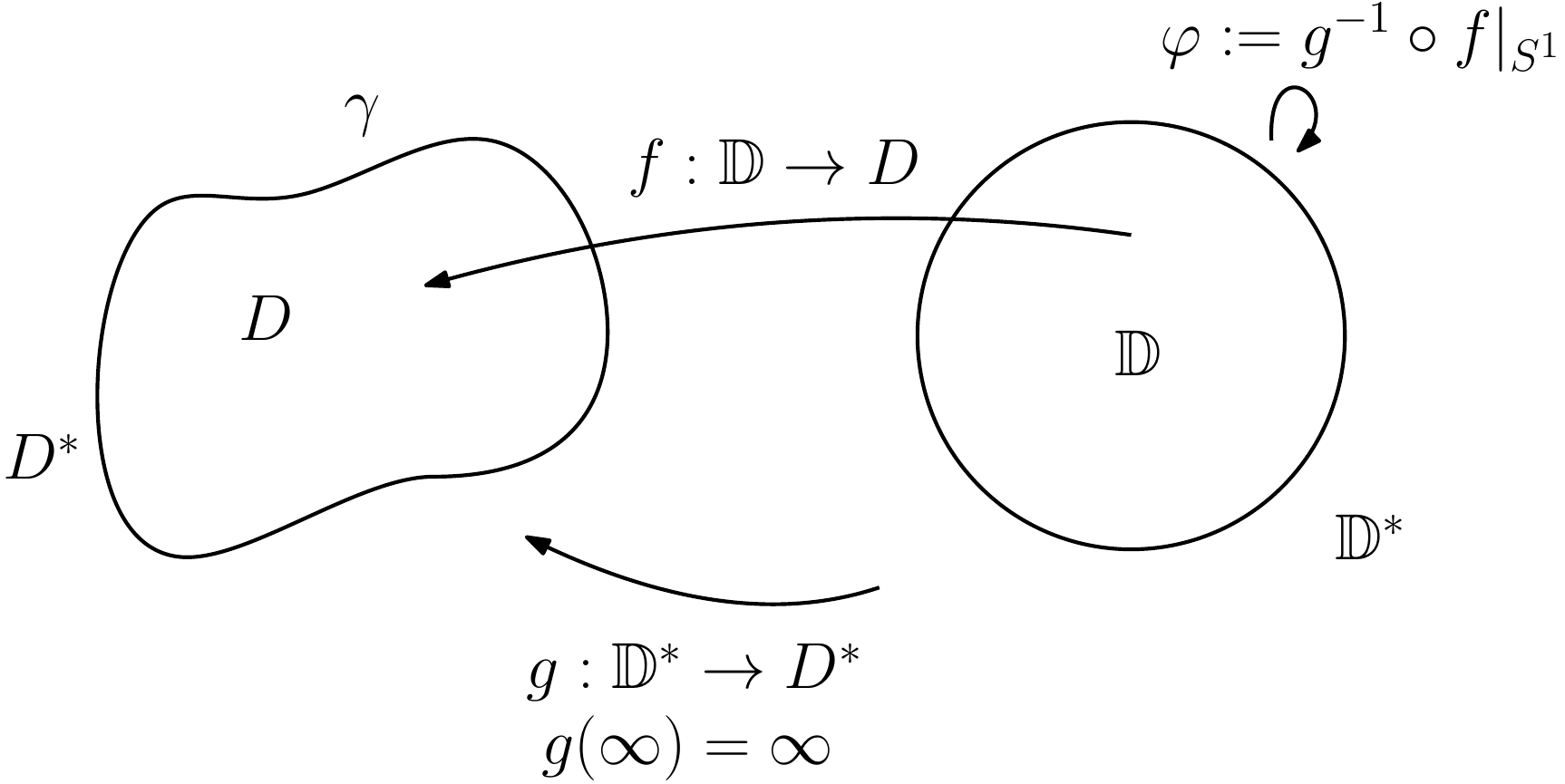}
 \caption{\label{fig_disk_welding} Welding function $\varphi$ of a simple loop $\g$. } 
 \end{figure}
We say that $\varphi: =  g^{-1} \circ f|_{S^1}$ is a \emph{welding function} of $\g$ (see Figure~\ref{fig_disk_welding}), which lies in QS$(S^1)$ as it is the boundary value of the quasiconformal map $g^{-1} \circ f$ on $\m D$ and does not depend on the extensions.

We say $\varphi \in \QS(S^1)$ is in the \emph{Weil-Petersson class} if $[\varphi] \in T_0(1)$, 
and $\g$ is a \emph{Weil-Petersson quasicircle} if its welding function $\varphi$ is in the Weil-Petersson class.
We define ${\bf S_1}(\g)$ to be 
$${\bf S} (f,g) : = \int_{\m D} \abs{\frac{f''}{f'}(z)}^2 \dd z^2 + \int_{\m D^*}\abs{\frac{g''}{g'}(z)}^2 \dd z^2  + 4\pi \log \abs{f'(0)} - 4\pi \log \abs{g'(\infty)},$$
which is finite if and only if $\g$ is a Weil-Petersson quasicircle and the value does not depend on the choice of $f$ and $g$.
In fact, for any other choice of conformal maps $\tilde f$ and $\tilde g$ for $\g$, there exists $\mu \in \text{M\"ob} (S^1)$ and $\nu \in \text{Rot} (S^1)$ such that $\tilde f = f \circ \mu$ and $\tilde g = g \circ \nu$.
It follows from explicit computations (\cite[Lem.~2.3.4]{TT2006WP}) that
$${\bf S} (f,g) = {\bf S} (\tilde f, \tilde g)$$
which is also equal to ${\bf S_1}([\varphi])$, see \cite[Lem.~2.3.4, Thm.~2.3.8]{TT2006WP}.

Now we can state the main theorem of this section: 
\begin{thm} \label{thm_energy_liouville}
Let $\g$ be a (bounded) Jordan curve, then $\g$ has finite Loewner energy if and only if $\g$ is a Weil-Petersson quasicircle. Moreover,
\begin{equation} \label{eq_energy_liouville}
I^L(\g) = \bf S_1(\g)/\pi.
\end{equation}
\end{thm}

It is worth mentioning other characterizations of $T_0(1)$ due to Cui, Shen, Takhtajan and Teo, from which  one obtains immediately other analytic characterizations of finite energy loops given Theorem~\ref{thm_energy_liouville}:
\begin{theorem}[\cite{Cui2000,Shen2013,TT2006WP}]
   With the same notation as in Theorem~\ref{thm_TT_equiv_T01}, $\varphi$ is in the Weil-Petersson class if and only if one of the following equivalent condition holds:
   \begin{enumerate}[(i)]
   \item $\varphi$ has quasiconformal extension to $\m D$, whose complex dilation $\mu = \partial_{\ad z} \varphi /\partial_{z} \varphi$ satisfies
   $$\int_{\m D} \abs{\mu(z)}^2 \rho(z) \dd z^2 <\infty;$$
   \item $\varphi$ is absolutely continuous with respect to arclength measure, such that $\log \varphi'$ belongs to the Sobolev space $H^{1/2}(S^1)$;
   \item the Grunsky operator associated to $f$ or $g$ is Hilbert-Schmidt. 
   \end{enumerate}
\end{theorem}

Now we proceed to the proof of Theorem~\ref{thm_energy_liouville}. We first prove it for smooth loops using results from Section~\ref{sec_detz}.
\begin {proof}[Proof for smooth loops]
Let $\g$ be a smooth (bounded) Jordan curve. 
It is clear from the definition that ${\bf S_1} (\g)$ is invariant under affine transformation of $\m C$. By M\"obius invariance of the Loewner loop energy, we may also assume that $\g$ is inside the Euclidean ball of radius $2$ and of center $0$.  

Let $g_0 = e^{2\psi} d z^2$ be a metric conformally equivalent to the Euclidean metric (or the spherical metric), such that $\psi \equiv 0$ on $B (0, 2)$ and $e^{2\psi (z)} = 4/(1 +\abs{z}^2)^2$ in a neighborhood of $\infty$ which makes $g_0$ coincide with the spherical metric near $\infty$. 
We compute the quotient on the right hand side of the expression in Theorem~\ref{thm_energy_determinant} \ref{item_energy_determinant} by taking $g = g_0$.

The same computation (and the same notations) as in the proof of Theorem~\ref{thm_energy_determinant} shows that
\begin{align*}
&12 \log \frac{\detz( - \D_{\m D, g_0}) \detz( - \D_{\m D^*, g_0})  }{\detz( - \D_{D, g_0}) \detz( - \D_{D^*, g_0})} \\
=& 
\frac{1}{ \pi} \left (\int_{\m D \cup \m D^*} \abs{\nabla_0 \s}^2 + 2K_0 \s \dd \vol_0 + \int_{S^1 \sqcup S^1} 2 k_0 \s \dd l_0 + 3 \partial_{n_0} \s \dd l_0 \right) \\
=&\frac{1}{ \pi}  \left (\int_{\m D} \abs{\nabla \s_f}^2 \dd z^2 +  \int_{\m D^*} \abs{\nabla \s_g}^2 \dd z^2 \right)+  \frac{2}{\pi} \int_{S^1\sqcup S^1} k_0 \s \dd l_0,
\end{align*}
where $\s = \s_f + \psi(f) -\psi$ for $z \in \ad{\m D}$, and $\s = \s_g + \psi(g) -\psi$ for $z \in \ad{\m D}^*$,
$S^1 \sqcup S^1$ denotes the two copies of $S^1$ as the boundary of $\m D$ and of $\m D^*$, the value of $\s$ on the boundary depends on the copy accordingly. 
In fact, the analogous sum \eqref{eq_vanish_det} of the second and the third term in \eqref{eq_three_terms} 
\begin{equation*}
  \frac{2}{\pi} \int_{S^1 \sqcup S^1}  -k \psi(f) - \s \partial_n \psi \dd l
   \end{equation*}
also vanishes here since $\psi$ is identically $0$ in a neighborhood of $S^1$ and of $\g$.  
The only difference with the proof of Theorem~\ref{thm_energy_determinant} is that we have an extra term (analogous to the last term in \eqref{eq_three_terms}): that is $ \int_{S^1\sqcup S^1} k_0 \s \dd l_0$ since $k_0$ is not vanishing: 
$k_0 (z) = 1 $ for $z\in \partial \m D$ and $k_0 (z) = -1 $ for $z\in \partial \m D^*$.
Using again the fact that $\psi(f(z)) = \psi(z) = 0$ for $z \in S^1$, the smoothness up to boundary and the harmonicity of $\s_f$ and $\s_g$, we get:
$$ \frac{2}{\pi} \int_{S^1\sqcup S^1} k_0 \s \dd l_0 =  4 \s_f(0) - 4 \s_g (\infty) = 4 \log \abs{f'(0)} - 4 \log \abs{g'(\infty)}.$$
Hence,
$$I^L(\g) = {\bf S_1}(\g) /\pi, $$
for the smooth loop $\g$ by Theorem~\ref{thm_energy_determinant}.
  \end{proof}

In particular, for a bounded smooth loop $\g \subset \m C$, we have the identity
\begin{equation} \label{eq_integral_identity}
J(h) =\frac{1}{\pi} \int_{\m C\setminus \mu(\g)} \abs{\frac{h''}{h'}(z)}^2 \dd z^2 = \frac{1}{\pi} {\bf S_1}(\g)
\end{equation}
where $\mu $ is a M\"obius function $\Chat \to \Chat$ such that $\mu(\g(0)) = \infty$, and $h$ is a conformal map from the complement of $\mu(\g)$ onto $\m H \cup \m H^*$ that fixes $\infty$, as defined in Theorem~\ref{thm_loop_identity}. 
The identity~\eqref{eq_integral_identity} of two domain integrals has \emph{a priori} no reason to depend on the boundary regularity, which then implies Theorem~\ref{thm_energy_liouville} for general loops by an approximation argument. 

To make the approximation precise, we will use the following lemma which characterizes the convergence in the \emph{universal Teichm\"uller curve} $\mc T (1)$ which is a complex fibration over $T(1)$, given by
\begin{align*}
 \text{Rot}(S^1) \backslash \QS(S^1) & \simeq \{\varphi \in \QS(S^1), \varphi(1) = 1\} \\
& \simeq \{f \in \mc U, f(0) = 0, f'(0) = 1, f \text{ admits q.c. extension to } \m C \}.
\end{align*}
The second identification is obtained from solving the conformal welding problem as for $T(1)$: for each $\varphi \in \QS(S^1)$ that fixes $1$, there exist unique conformal maps $f$ and $g$ on $\m D$ and $\m D^*$ (\emph{canonically associated to $\varphi \in \mc T(1)$}), which satisfy \ref{CW_ext} and \ref{CW_g_infty} and 
\begin{enumerate}[CW'1.]
\setcounter{enumi}{1}
\item \label{CW'_welding} $\varphi = g^{-1} \circ f|_{S^1}$.
\item \label{CW'_f} $f(0) = 0$, $f'(0) = 1$.
\end{enumerate}
 Let $\pi : \mc T (1) \to T(1)$ be the projection and $\mc T_0(1) : = \pi^{-1} (T_0(1))$ is also a Hilbert manifold such that $\pi$ is fibration of Hilbert manifolds, see \cite[Appx.A]{TT2006WP}.

\begin{lemA}[{\cite[Cor.~A.4, Cor.~A.6]{TT2006WP}}]
\label{lem_TT_cor} 
Let $\{\varphi_n\}_{n= 1}^{\infty}$ be a sequence of points in $\mc T_0(1)$, let $f_n$ and $g_n$ be the conformal maps canonically associated to $\varphi_n$ such that $\varphi_n = g_n^{-1} \circ f_n$, and similarly let $\varphi = g^{-1} \circ f \in \mc T_0(1)$. Then the following conditions are equivalent:
\begin{enumerate}
  \item In $\mc T_0(1)$ topology, 
  $$\lim_{n\to \infty} \varphi_n = \varphi.$$
  \item 
  $$\lim_{n\to \infty} \int_{\m D} \abs{\frac{f_n''}{f_n'} (z) - \frac{f''}{f'} (z)}^2 \dd z^2 = 0. $$
  \item Let  $\tilde g(z) := 1/g(1/z)$ and $\tilde g_n(z) := 1/g_n(1/z)$ for all $n \geq 1$,
$$\lim_{n\to \infty} \int_{\m D} \abs{\frac{\tilde g_n''}{\tilde g_n'} (z) - \frac{\tilde g''}{\tilde g'} (z)}^2 \dd z^2 = 0. $$
\end{enumerate}
If above conditions are satisfied, then we have also
 $$\lim_{n\to \infty} \int_{\m D^*} \abs{\frac{g_n''}{g_n'} (z) - \frac{g''}{g'} (z)}^2 \dd z^2 = 0, $$
 and 
 $$ \lim_{n\to \infty}  {\bf S_1} ([\varphi_n]) =  \lim_{n\to \infty}  {\bf S} (f_n, g_n) =   {\bf S} (f,g) = {\bf S_1} ([\varphi]) .$$
\end{lemA}

We will also use the following lemma on the lower-semicontinuity of ${\bf S_1}$:
\begin{lem}\label{lem_lower_semi_S1} If a sequence $(\g_n: [0,1] \to \Chat)_{n \geq 0}$ of simple loops converges uniformly to a bounded loop $\g$, then
$$\liminf_{n \to \infty} {\bf S_1}(\g_n) \geq  {\bf S_1}(\g). $$
\end{lem}
\begin{proof}
There is an $n_0$ large enough, such that $(\g_n)_{n \ge n_0}$ are bounded and $\cap_{ n \geq n_0} D_n \neq \emptyset$  where $D_n$ denotes the bounded connected component of $\m C \setminus \g_n$. Let
$z_0 \in \cap_{ n \geq n_0} D_n,$
and for $n \geq n_0$, $f_n: \m D \to D_n$ a conformal map such that $f_n (0) = z_0$ and $f_n'(0) > 0$.

From the Carath\'eodory kernel theorem, $f_n$ converges uniformly on compacts to $f: \m D \to D$, where $D$ is the bounded connected component of $\Chat\setminus \g$. It yields that for $K \subset \m D$ compact set,
\begin{align*}
\liminf_{n \to \infty} \int_{\m D} \abs{\frac{f_n''}{f_n'} (z)}^2 \dd z^2 \geq \liminf_{n\to \infty} \int_{K} \abs{\frac{f_n''}{f_n'} (z)}^2 \dd z^2 = \int_{K} \abs{\frac{f''}{f'} (z)}^2 \dd z^2.
\end{align*}
Since $K $ is arbitrary, 
$$\liminf_{n \to \infty} \int_{\m D} \abs{\frac{f_n''}{f_n'} (z)}^2 \dd z^2 \geq \int_{\m D} \abs{\frac{f''}{f'} (z)}^2 \dd z^2. $$
Similarly, let $g_n$ be the conformal map from $\m D^*$ onto the unbounded connected component $D_n^*$ of $\m C \setminus \g_n$ and $g:  \m D^* \to D^*$ that fix $\infty$. We have also that $g_n$ converges locally uniformly on compacts to $g$, and   
$$\liminf_{n \to \infty} \int_{\m D^*} \abs{\frac{g_n''}{g_n'} (z)}^2 \dd z^2 \geq \int_{\m D^*} \abs{\frac{g''}{g'} (z)}^2 \dd z^2.$$
And we have also $g_n'(\infty) \to g'(\infty)$, $f_n'(0) \to f'(0)$.
Hence 
$$\liminf_{n \to \infty} {\bf S_1}(\g_n) = \liminf_{n \to \infty} {\bf S}(f_n, g_n) \geq  {\bf S}(f, g) ={\bf S_1}(\g) $$
as we claimed.
  \end{proof}

We also cite the similar lower-semicontinuity of the Loewner loop energy from \cite{LoopEnergy}:
with the same condition,
$$\liminf_{n\to \infty} I^L(\g_n, \g_n(0)) \geq I^L(\g, \g(0)).$$

We can now finally prove Theorem~\ref{thm_energy_liouville} in the general case using approximations by smooth loops.
\begin{proof}[Proof for general loops]
Assume that ${\bf S_1}(\g) < \infty$. Let $f: \m D \to D$ and $g: \m D^* \to D^*$ be conformal maps associated to $\g$, without loss of generality we may assume that $f(0) = 0$, $f'(0) = 1$ and $g(\infty) = \infty$, so that $(f,g)$ is canonically associated to $g^{-1} \circ f \in \mc T_0(1)$.
Consider the sequence  $\g^n  := f(c_n S^1)$ of smooth loops that converges uniformly as parametrized loop (by $S^1$) to $\g$, where $c_n \uparrow 1$. Let $f_n (z) : = c_n^{-1} f(c_n z)$ such that $f_n(0) = 0$ and $f_n'(0) = 1$.
It is not hard to see that 
$$\lim_{n\to \infty} \int_{\m D} \abs{\frac{f_n''}{f_n'} (z) - \frac{f''}{f'} (z)}^2 \dd z^2 = 0. $$
In fact, $f_n$ converges uniformly to $f$ on $(1-\vare) 
\m D$ for $\vare >0$. And the above integral on the annulus $\m D \setminus (1-\vare) \m D$ is arbitrarily small as $\vare\to 0$ since $\bf{S_1} (\g)$ is finite.

Hence by Lemma~\ref{lem_TT_cor},
${\bf S_1}(\g^n)$ converges to ${\bf S_1}(\g)$. 
Since $\g^n$ converges uniformly to $\g$, from the lower-semicontinuity of Loewner energy and Theorem~\ref{thm_energy_liouville} for smooth loops,
\begin{equation} \label{eq_S_larger_J}
{\bf S_1}(\g) /\pi = \liminf_{n\to \infty} {\bf S_1}(\g^n) /\pi = \liminf_{n\to \infty} I^L(\g^n) \geq I^L(\g).
\end{equation}

Similarly, assume now that $I^L(\g) <\infty$ with driving function $W : \m R \to \m R$. Without loss of generality, we assume also that $\g$  is bounded and passes through $-1, -i, 1$. 
Let $W_n \in C_0^{\infty}(\m R)$ be a sequence of compactly supported smooth function, such that 
$$ \int_{-\infty}^{\infty} \abs{W' (t)- W'_n (t)}^2 \dd t \xrightarrow[]{n \to \infty} 0. $$
Let $\g_n$ be a loop in $\S$ with driving function $W_n$. By \cite{lind-tran}, $\g_n$ is smooth. We may assume that 
$
\sup_{n \geq 1} I^L(\g_n) < \infty
$
and $\g_n$ passes through $-1, -i, 1$ as well.
By \cite[Prop.~2.9]{LoopEnergy}, there exists $K > 1$ such that  $\g$ and $\g_n$ are $K$-quasicircles.
The compactness of $K$-quasiconformal maps allows us to subtract a subsequence $\g_{n_k}$ that converges uniformly to $\g$.

From Theorem~\ref{thm_energy_liouville} for smooth loops $\g_n$ and Lemma~\ref{lem_lower_semi_S1}, we have
$$I^L(\g) = \liminf_{k \to \infty} I^L(\g_{n_k}) = \liminf_{k \to \infty} {\bf S_1}(\g_{n_k})  /\pi \geq  {\bf S_1}(\g) /\pi.$$
We conclude that $I^L(\g)< \infty$ if and only if ${\bf S_1}(\g) < \infty$ and $I(\g) = {\bf S_1} (\g) / \pi$ as claimed in Theorem~\ref{thm_energy_liouville}.
  \end{proof}

   \section {An informal discussion}
   \label {informal}

Let us conclude with some very loose comments on the relation
between our Theorem~\ref{thm_loop_identity} and the theory of SLE and Liouville quantum gravity (LQG).  
Recall first that the Loewner energy was shown in \cite{wang_loewner_energy} to be
a large deviation rate function of SLE$_\k$ as $\k $ goes to $0$. Heuristically,
 \[I(\g) = \lim_{\vare \to 0} \lim_{\k \to 0} -\k \log \P (\text{SLE}_\k \text{ stays } \vare\text{-close to } \g).  \]
Given a sufficiently smooth simple curve $\g$, the mapping-out function $h$ from the complement of $\g$ to a standard domain ($\m H \cup \m H^*$), induces a metric on the standard domain that is the push-forward of the Euclidean metric of the initial domain. 
The exponential exponent of the conformal factor is given by $\s_{h^{-1}} (\cdot) := \log \abs{h^{-1} (\cdot)'}$. 
It prescribes in turn the welding homeomorphism of the curve $\g$ on $\m R$ by identifying boundary points according to the boundary length of this metric (see Figure~\ref{fig_welding}).  

 \begin{figure}[ht]
 \centering
 \includegraphics[width=0.7\textwidth]{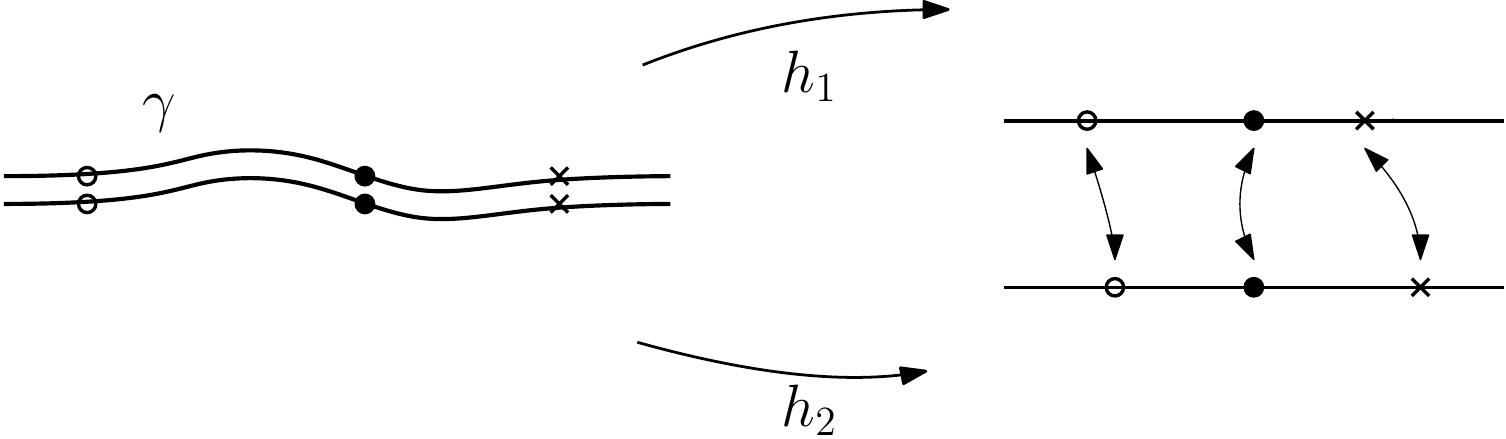}
 \caption{\label{fig_welding} Welding of a simple loop $\g$ passing through $\infty$. } 
 \end{figure}

On the other hand, the LQG approach to SLE pioneered by Sheffield in \cite{scott_zipper} provides an interpretation of SLE curves via welding of structures defined using the exponential of the Gaussian Free Field (GFF). 
More specifically, let $\Phi$ be a free boundary Gaussian free field on the standard domain. That is the random field that can be described in loose term 
as having a ``density'' proportional to 
$$\exp\left(-\frac{1}{4\pi} \int \abs{\nabla \Phi (z)}^2 \dd z^2\right).$$ 
One takes formally $\exp(\sqrt{\k} \Phi)$ times the Lebesgue measure 
(modulo some appropriate renormalization procedure) to define a random measure (LQG) on the standard domain (which corresponds in fact to $\sqrt \k$-quantum wedges with an opening angle $\t$ which converges to $\pi$ when $\k \to 0$). It also induces a random boundary length which can be viewed as 
$\exp((\sqrt{\k}/2 )\Phi)$ times the Euclidean arclength (again modulo some appropriate renormalization procedure).  
Intuitively, the quantum zipper then states that welding two independent free boundary GFFs up according to their random boundary length gives an SLE$_\k$ curve.

We can note that the Dirichlet energy of $\s_{h^{-1}}$ is the action functional that is naturally associated to the Gaussian free field, so that 
in a certain sense, one has a large deviation principle of the type
\begin{eqnarray*}
 \lefteqn { \lim_{\vare \to 0} \lim_{\k \to 0} -\k \log \P ( (\sqrt \k/2) \Phi \text{ stays } \vare \text{-close to } \s_{h^{-1}})} \\
 & \approx & \lim_{\k \to 0} - \k \log \exp\left(-\frac{1}{4\pi} \int \abs{\frac{2 \nabla  \s_{h^{-1}}}{\sqrt \k}}^2 \dd z^2\right)\\
 & = & \frac{1}{\pi} \int \Big |\nabla \s_{h^{-1}} (z)\Big |^2 \dd z^2.
\end{eqnarray*}
Hence, our identity between the Loewner energy and the Dirichlet energy of $\sigma_h$ (which is the same as the Dirichlet energy of $\sigma_{h^{-1}}$) 
is loosely speaking equivalent to the fact that (in some sense) as $\kappa \to 0$ (and then $\eps \to 0$), 
the decay rates of $$\P ( (\sqrt \k/2) \Phi \text{ stays } \vare \text{-close to } \s_{h^{-1}}) \hbox { and } 
\P (\text{SLE}_\k \text{ stays } \vare\text{-close to } \g)$$
are comparable. 
However, the above argument is not even close to be rigorous (it would be interesting to explore it though) 
and the proof in this paper follows a completely different route and does not use any knowledge about SLE, LQG or the quantum zipper.

\appendix

\section{Geodesic curvature formula}
As many of our proofs rely on the following formula on the variation of the geodesic curvature under a Weyl-scaling, we sketch a short proof for readers' convenience.
\begin{lem}\label{lem:geodesic_curvature}
  Let $(D, g_0)$ be a surface with smooth boundary $\g = \partial D$. If  $\s \in C^{\infty} (D, \m R)$ and $g = e^{2 \s} g_0$, the geodesic curvature of $ \partial D$ under the metric $g$ satisfies
  $$k_g =  e^{-\s} \left( k_0  + \partial_{n_0} \s \right),$$
  where $k_0$ is the geodesic curvature under the metric $g_0$, and $\partial_{n_0}$ the outer-normal derivative with respect to to $g_0$.
\end{lem}

\begin{proof}
 We parameterize $\g$ by arclength in $g$ and let $N$ be the outer normal vector field on~$\g$,
 namely  
 $g (\dot \g , \dot \g ) = g (N, N) \equiv 1$. 
 We have that $\dot \g_0 :=  e^{\s} \dot \g $ and $N_0 : = e^{\s} N$ are unit vectors under $g_0$.
 The geodesic curvature of $\partial D$  is given by 
 $$k_g = g \left(\nabla_{g,\dot \g }\dot \g , - N \right). $$
 The covariant derivative $\nabla_g$ is related to the covariant derivative $\nabla_0$ under $g_0$ by 
 $$\nabla_{g, X} Y = \nabla_{0, X} Y + X(\s) Y + Y(\s) X - g_0 (X,Y) \nabla_0 \s.$$
 Therefore,
 \begin{align*}
 \nabla_{g,\dot \g}\dot \g &=  \nabla_{0, \dot \g } \dot \g + 2g_0 (\dot \g, \nabla_0 \s) \dot \g - g_0 (\dot \g ,\dot \g) \nabla_0 \s \\
 &= e^{-2\s}  \nabla_{0, \dot \g_0 } \dot \g_0 +  2g_0 (\dot \g, \nabla_0 \s) \dot \g - e^{-2\s} \nabla_0 \s.
 \end{align*}
 Since $g (\dot \g, N) = 0$, we have
 \begin{align*}
 g \left(\nabla_{g,\dot \g }\dot \g, - N \right) & =  e^{2\s} g_0 \left(e^{-2\s} (\nabla_{0, \dot \g_0 } \dot \g_0  - \nabla_0 \s), -e^{-\s} N_0\right) \\
 & = e^{-\s } (k_0  + \partial_{n_0} \s )
 \end{align*}
 as claimed.
\end{proof}

\end{document}